\documentclass[a4paper,11pt]{article}
\usepackage[utf8]{inputenc}

\usepackage{graphicx}
\usepackage{geometry}
\usepackage{amsthm}
\usepackage{amssymb}
\usepackage{amsmath}
\usepackage{hyperref}
\usepackage{xcolor}
\usepackage{enumerate}
\usepackage{listings}
\usepackage{complexity}
\usepackage{microtype}
\usepackage{csquotes}

\usepackage[font=small,labelfont=bf]{caption}
\usepackage[font=small,labelfont=normalfont,labelformat=simple]{subcaption}

\graphicspath{{figs/}}

\newtheorem{theorem}{Theorem}

\newtheorem{lemma}[theorem]{Lemma}
\newtheorem{proposition}[theorem]{Proposition}
\newtheorem{corollary}[theorem]{Corollary}
\newtheorem{problem}{Problem}

\DeclareMathOperator{\conv}{conv}

\def\RR{\mathbb{R}}

\usepackage{tikz}
\usepackage{mathtools}
\newcommand{\myarrow}[1][0]{%
  \mathrel{%
    \text{$
     \begin{tikzpicture}[baseline = -0.5ex]
       \node[inner sep=0pt,outer sep=0pt,rotate = #1] (a) at (0,0)  {$\xrightarrow{}$};
    \end{tikzpicture}
    $}%
  }%
}%

\def\inst#1{$^{#1}$}

\date{}

\title{Tight bounds on the expected number of holes \\in random point sets}

\begin{document}

\author{Martin Balko\inst{1}\thanks{M. Balko was supported by the grant no.~21-32817S of the Czech Science Foundation (GA\v{C}R), by the Center for Foundations of Modern Computer Science (Charles University project UNCE/SCI/004), and by the PRIMUS/17/SCI/3 project of Charles University.
This article is part of a project that has received funding from the European Research Council (ERC) under the European Union's Horizon 2020 research and innovation programme (grant agreement No 810115).
} 
\and
Manfred Scheucher\inst{2}\thanks{M.\ Scheucher was supported by the DFG Grants SCHE~2214/1-1 and FE~340/12-1, and by the internal research funding ``Post-Doc-Funding'' from Technische Universit\"at Berlin. 
We also acknowledge support by the internal research program IFFP 2016--2020 of the FernUniversi\"at in Hagen.
}
\and
Pavel Valtr\inst{1}\thanks{P. Valtr was supported by the grant no.~21-32817S of the Czech Science Foundation (GA\v{C}R) and by the PRIMUS/17/SCI/3 project of Charles University.}
}

\maketitle

\begin{center}
{\footnotesize
\inst{1} 
Department of Applied Mathematics, \\
Faculty of Mathematics and Physics, Charles University, Czech Republic \\
\texttt{balko@kam.mff.cuni.cz}
\\\ \\
\inst{2} 
Institut f\"ur Mathematik, Technische Universit\"at Berlin, Germany\\
\texttt{scheucher@math.tu-berlin.de}
}
\end{center}

\begin{abstract}
For integers $d \geq 2$ and $k \geq d+1$, a \emph{$k$-hole} in a set $S$ of points in general position in $\mathbb{R}^d$ is a $k$-tuple of points from $S$ in convex position such that the interior of their convex hull does not contain any point from $S$.
For a convex body $K \subseteq \mathbb{R}^d$ of unit $d$-dimensional volume, we study the expected number $EH^K_{d,k}(n)$ of $k$-holes in a set of $n$ points drawn uniformly and independently at random from $K$. 

We prove an asymptotically tight lower bound on $EH^K_{d,k}(n)$ by showing that, for all fixed integers $d \geq 2$ and  $k\geq d+1$, the number $EH_{d,k}^K(n)$ is at least $\Omega(n^d)$.
For some small holes, we even determine the leading constant $\lim_{n \to \infty}n^{-d}EH^K_{d,k}(n)$ exactly.
We improve the currently best known lower bound on $\lim_{n \to \infty}n^{-d}EH^K_{d,d+1}(n)$ by Reitzner and Temesvari~(2019). 
In the plane, we show that the constant $\lim_{n \to \infty}n^{-2}EH^K_{2,k}(n)$ is independent of $K$ for every fixed $k \geq 3$ and we compute it exactly for $k=4$, improving earlier estimates by Fabila-Monroy, Huemer, and Mitsche~(2015) and by the authors~(2020).
\end{abstract}

\section{Introduction}

For a positive integer $d$, let $S$ be a set of points from $\mathbb{R}^d$ in \emph{general position}.
That is, for every $k \in \{1,\dots,d-1\}$, no $k+2$ points from $S$ lie on a $k$-dimensional affine subspace of~$\mathbb{R}^d$.
Throughout the whole paper we only consider point sets in~$\mathbb{R}^d$ that are finite and in general position.

A point set $P$ is in \emph{convex position} if no point from $P$ is contained in the convex hull of the remaining points from $P$.
For an integer $k \geq d+1$, a \emph{$k$-hole} $H$ in $S$ is a set of $k$ points from $S$ in convex position such that the convex hull ${\rm conv}(H)$ of $H$ does not contain any point of $S$ in its interior.

The study of $k$-holes in point sets was initiated by Erd\H{o}s~\cite{Erdos1978} who asked whether, for each $k \in \mathbb{N}$, every sufficiently large point set in the plane contains a $k$-hole.
This was known to be true for $k \leq 5$, but, in the 1980s, Horton~\cite{Horton1983} constructed arbitrarily large point sets without $7$-holes.
The question about the existence of $6$-holes was a longstanding open problem until 2007, when Gerken~\cite{Gerken2008} and Nicolas~\cite{Nicolas2007} showed that every sufficiently large set of points in the plane contains a $6$-hole.

The existence of $k$-holes was considered also in higher dimensions.
Valtr~\cite{VALTR1992b} showed that, for $k \le 2d+1$, every sufficiently large set of points in $\mathbb{R}^d$ contains a $k$-hole.
He also constructed arbitrarily large sets of points in $\mathbb{R}^d$ that do not contain any $k$-hole with $k > 2^{d-1} (P(d-1)+1)$,  where $P(d-1)$ denotes the product of the first $d-1$ prime numbers.
Very recently Bukh, Chao, and Holzman~\cite{bch20} improved this construction and Colon and Li~\cite{ConlonLim2021} construted new sets in $\mathbb{R}^d$ with no large holes.

Estimating the number of $k$-holes in point sets in $\mathbb{R}^d$ attracted a lot of attention; see~\cite{5holes_journal_version}.
In particular, it is well-known that the minimum number of $(d+1)$-holes  (also called \emph{empty simplices}) in sets of $n$ points in $\mathbb{R}^d$ is of order $O(n^d)$.
This is tight, as every set of $n$ points in $\mathbb{R}^d$ contains at least $\binom{n-1}{d}$ $(d+1)$-holes~\cite{BaranyFueredi1987,katMe88}.

The tight upper bound $O(n^d)$ can be obtained by considering random point sets drawn from a convex body.
More formally, a \emph{convex body} in $\mathbb{R}^d$ is a compact convex subset of $\mathbb{R}^d$ with a nonempty interior.
We use $\lambda_d$ to denote the $d$-dimensional Lebesgue measure on $\mathbb{R}^d$ and $\mathcal{K}_d$ to denote the set of all convex bodies in $\mathbb{R}^d$ of volume $\lambda_d(K)=1$.
For an integer $k \geq d+1$ and a convex body $K \in \mathcal{K}_d$, let $EH^K_{d,k}(n)$ be the expected number of $k$-holes in a set $S$ of $n$ points chosen uniformly and independently at random from $K$. 
Note that $S$ is in general position with probability $1$.

B\'{a}r\'{a}ny and F\"{u}redi~\cite{BaranyFueredi1987} proved the upper bound $EH_{d,d+1}^K(n) \leq (2d)^{2d^2} \cdot \binom{n}{d}$ for every $K \in \mathcal{K}_d$.
Valtr~\cite{Valtr1995} improved this bound in the plane by showing $EH^K_{2,3}(n) \le 4\binom{n}{2}$ for any $K \in \mathcal{K}_2$.
Very recently, Reitzner and Temesvari~\cite[Theorem~1.4]{ReitznerTemesvari2019} 
showed that this bound on $EH_{2,3}^K(n)$ is asymptotically tight for every  $K \in \mathcal{K}_2$.
This follows from their more general bounds
\begin{equation}
\label{eq-reitTemPlane}
\lim_{n \to \infty} n^{-2}EH_{2,3}^K(n) = 2
\end{equation}
and
\begin{equation}
\label{eq-reitTem}
\frac{2}{d!} \leq \lim_{n \to \infty} n^{-d}EH_{d,d+1}^K(n) \leq \frac{d}{(d+1)}\frac{\kappa^{d+1}_{d-1}\kappa_{d^2}}{\kappa^{d-1}_d\kappa_{(d-1)(d+1)}}
\end{equation}
for $d \geq 2$, where $\kappa_d = \pi^{\frac{d}{2}}\Gamma(\frac{d}{2}+1)^{-1}$ is the volume of the $d$-dimensional Euclidean unit ball.
Moreover, the upper bound in~\eqref{eq-reitTem} holds with equality in the case $d = 2$, and if $K$ is a $d$-dimensional ellipsoid with $d \geq 3$.
Note that, by \eqref{eq-reitTem}, 
there are absolute positive constants $c_1,c_2$ such that 
\[
d^{-c_1d} \le \lim_{n \to \infty}n^{-d}EH_{d,d+1}^K(n) \le d^{-c_2d}
\]
for every $d \geq 2$ and $K \in \mathcal{K}_d$.

Considering general $k$-holes in random point sets in $\mathbb{R}^d$, 
the authors~\cite{bsv2021} recently proved that $EH_{d,k}^K(n) \leq O(n^d)$ for all fixed integers $d \geq 2$ and $k \geq  d+1$ and every $K \in \mathcal{K}_d$.
More  precisely, we showed 
\begin{equation}
\label{eq-holesUpperBound}
EH_{d,k}^K(n) \leq 2^{d-1}\cdot \left(2d^{2d-1}\binom{k}{\lfloor d/2 \rfloor}\right)^{k-d-1} \cdot \frac{n(n-1) \cdots (n-k+2)}{(k-d-1)! \cdot (n-k+1)^{k-d-1}}.
\end{equation}

In this paper, we also study the expected number $EH_{d,k}^K(n)$ of $k$-holes in random sets of $n$ points in $K$.
In particular, we derive a lower bound that asymptotically matches the upper bound~\eqref{eq-holesUpperBound} for all fixed values of $k$.
Moreover, for some small holes, we even determine the leading constants $\lim_{n \to \infty}n^{-d}EH_{d,k}^K(n)$.

\section{Our Results}
\label{sec:results}

First, we show that for all fixed integers $d \geq 2$ and  $k\geq d+1$ the number $EH_{d,k}^K(n)$ is in~$\Omega(n^d)$, which matches the upper bound~\eqref{eq-holesUpperBound} by the authors~\cite{bsv2021} up to the leading constant.

\begin{theorem} 
\label{thm:largerHoles}
For all integers $d \geq 2$ and $k \geq d+1$, there are constants $C=C(d,k)>0$ and $n_0=n_0(d,k)$ such that, for every integer  $n \ge n_0$ and every convex body $K \subseteq \mathbb{R}^d$ of unit volume, we have $EH^K_{d,k}(n) \geq C \cdot n^d$.
\end{theorem}

In particular, we see that random point sets typically contain many $k$-holes no matter how large $k$ is, as long as it is fixed.
This contrasts with the fact that, for every $d \geq 2$, there is a number $t=t(d)$ and arbitrarily large sets of points in $\mathbb{R}^d$ without any $t$-holes~\cite{Horton1983,VALTR1992b}.

Theorem~\ref{thm:largerHoles} together with~\eqref{eq-holesUpperBound} shows that $EH^K_{d,k}(n)=\Theta(n^d)$ for all fixed integers $d$ and $k$ and every $K \in \mathcal{K}_d$, which determines the asymptotic growth rate of $EH^K_{d,k}(n)$.
We thus focus on determining the leading constants $\lim_{n \to \infty}n^{-d}EH^K_{d,k}(n)$ for small holes.

For a convex body $K \subseteq \mathbb{R}^d$ (of a not necessarily unit volume), we use $p^K_d$ to denote the probability that 
the convex hull of $d+2$ points chosen uniformly and independently at random from $K$ 
is a $d$-simplex.
That is, the probability that one of the $d+2$ points falls in the convex hull of the remaining $d+1$ points.
The problem of computing $p^K_d$ is known as the $d$-dimensional \emph{Sylvester’s convex hull problem} for $K$ and it has been studied extensively.
Let $p_d = \max_K p^K_d$, where the maximum is taken over all convex bodies $K \subseteq \mathbb{R}^d$. 
We note that the maximum is achieved, since it is well-known that every affine-invariant
continuous functional on the space of convex bodies attains a maximum.

First, we prove the following lower bound on the expected number $EH_{d,d+1}^K(n)$ of empty simplices in random sets of $n$ points in $K$, which improves the lower bound from~\eqref{eq-reitTem} by Reitzner and Temesvari~\cite{ReitznerTemesvari2019} by a factor of $d/p_{d-1}$.

\begin{theorem}
\label{thm:improvedLowerBound}
For every integer $d \geq 2$ and every convex body $K \subseteq \mathbb{R}^d$ of unit volume, we have
\[
\lim_{n \to \infty} n^{-d}EH_{d,d+1}^K(n) \ge \frac{2}{(d-1)!p_{d-1}}.
\]
\end{theorem}

Using the trivial fact $p_1 = 1$ with the inequality $EH_{2,3}^K(n) \leq 2(1+o(1))n^2$ proved by Valtr~\cite{Valtr1995}, we see that the leading constant in our estimate is asymptotically tight in the planar case.
Results of Blaschke~\cite{Blaschke1917,Blaschke1923} (see Theorem~\ref{thm:Blaschke}) and Groemer~\cite{Groemer1973} (see Theorem~\ref{thm:Groemer}) imply the following bounds showing that, in any dimension $d \ge 3$, the leading constant in $EH_{d,d+1}^K(n)$ depends on the convex body.

\begin{corollary}
\label{cor:improvedLowerBound}
For every $d \geq 3$, 
every $d$-dimensional simplex $S$ of unit volume
fulfills
\[\lim_{n \to \infty} n^{-d}EH^S_{d,d+1}(n) < \lim_{n \to \infty} n^{-d}EH^B_{d,d+1}(n),\]
where $B$ is a $d$-dimensional ball of unit volume.
Moreover, 
for every convex body $K \subseteq \mathbb{R}^3$ of unit volume,
we have
\[
3 \le \lim_{n \to \infty} n^{-3}EH_{3,4}^K(n) \le \frac{12 \pi^2}{35} \approx 3.38386. 
\]
\end{corollary}

We  performed computer experiments 
to obtain estimates 
on the constants $c_{3,4}^K=\lim_{n \to \infty} n^{-3}EH_{3,4}^K(n)$ 
when $K$ is a Platonic solid or a ball.
The idea was to repeatedly sample 3 points $p_1,p_2,p_3$ uniformly at random from~$K$ and to compute the ratio of the so-called \emph{Fischer triangle} spanned by $p_1,p_2,p_3$ which is contained inside~$K$.
This ratio coincides with $c_{3,4}^K$; see Section~3 in \cite{bsv2021}.
Since this ratio is a realization of an independent and identically distributed random variable,
the sample average is normally distributed by the central limit theorem.
After evaluating $16$ million samples by our program,
we now have a statistical evidence 
that 
\begin{itemize}
\item 
$3.095 \leq c_{3,4}^T \leq 3.097$ for a regular tetrahedron~$T$,
\item 
$3.264 \leq c_{3,4}^C \leq 3.266$ for a cube~$C$,
\item 
$3.289 \leq c_{3,4}^O \leq 3.291$ for a regular octahedron~$O$,
\item 
$3.354 \leq c_{3,4}^D \leq 3.356$ for a regular dodecahedron~$D$,
and
\item 
$3.362 \leq c_{3,4}^I \leq 3.364$ for a regular icosahedron~$I$.
\end{itemize}
Moreover, our program gave a statistical evidence that $3.383 \leq c_{3,4}^B \leq 3.385$ for a ball~$B$,
which for dimension $2$ confirms the constant by Reitzner and Temesvari \cite{ReitznerTemesvari2019};
see the upper bound in \eqref{eq-reitTem}.
All the bounds are with confidence level $99.99 \%$.
The source code of our program and more information is available on the supplemental website~\cite{supplemental_program}.

By Theorem~\ref{thm:improvedLowerBound}, better upper bounds on $p_{d-1}$ imply stronger lower bounds on $EH^K_{d,d+1}(n)$.
The problem of estimating $p_d$ is equivalent to the problem of estimating the expected $d$-dimensional volume $EV_d^K$ of the convex hull of $d+1$ points drawn from a convex body $K \subseteq \mathbb{R}^d$ uniformly and independently at random, since $p^{K}_d = \frac{(d+2)EV_d^K}{\lambda_d(K)}$; see~\cite{klee69,schneider88}.
In the plane, Blaschke~\cite{Blaschke1917,Blaschke1923} showed that $EV^K_2$ is maximized if $K$ is a triangle (see Theorem~\ref{thm:Blaschke}), which we use in Section~\ref{sec:proof_thm_improvedLowerBound} to derive the lower bound in Corollary~\ref{cor:improvedLowerBound}.
For $d \geq 3$, it is one of the major problems in convex geometry to decide whether $EV^K_d$ is maximized if $K$ is a simplex~\cite{SchneiderWeil2008}.
Meckes~\cite{meckes04} proved that there is a constant $c>0$ such that every convex body $K\in \mathcal{K}_d$ with $d \geq 2$ satisfies
\[EV^K_d \leq (cd^{-1/4}\log{d})^d,\]
which by Theorem~\ref{thm:improvedLowerBound} implies the following improved lower bound on $EH^K_{d,d+1}(n)$.

\begin{corollary}
\label{cor-newLowerBound}
There is a constant $c>0$ such that for every integer $d \geq 2$ and every convex body $K \subseteq \mathbb{R}^d$ of unit volume, we have
\[\lim_{n \to \infty}n^{-2}EH^K_{d,d+1}(n) \geq \frac{2}{(d+2)(d-1)!}\left(\frac{d^{1/4}}{c\log{d}}\right)^d.\]
\end{corollary}

We also note that there are connections between the probabilities $p_d$, 
the famous \emph{Slicing conjecture} 
(also known as \emph{Hyperplane conjecture}), 
and the \emph{isotropic constant}; 
see for example~\cite{meckes04}.

For convex bodies $K$ with small diameter, we can obtain the following better estimate on $p^K_d$.

\begin{proposition}
\label{prop:improvedLowerBound2}
Let $\varepsilon > 0$ be a real number and let $d \geq 1$ be an integer.
Let $K \subseteq \mathbb{R}^d$ be a convex body.
If there is a volume-preserving affine transformation $f\colon \mathbb{R}^d \to \mathbb{R}^d$ such that $f(K)$ has diameter $D$, then $p_d^K \leq\frac{ (d+2)D^d}{d!\lambda_d(K)}$.
In particular, if $D \leq d^{1-\varepsilon}\lambda_d(K)^{1/d}$, then $p^K_d \leq \frac{(d+2)d^{(1-\varepsilon)d}}{d!}$.
\end{proposition}

We note that there are convex bodies that do not satisfy the assumption from Proposition~\ref{prop:improvedLowerBound2}, for example the regular $d$-dimensional simplex.

Besides empty simplices, we also consider larger $k$-holes.
The expected number $EH_{2,4}^K(n)$ of $4$-holes in random planar sets of $n$ points was considered by Fabila-Monroy, Huemer, and Mitsche~\cite{mhm15}, who showed
\[
EH_{2,4}^K(n) \leq 18\pi D^2 n^2 + o(n^2)
\]
for any $K \in \mathcal{K}_2$, where $D=D(K)$ is the diameter of $K$.
Since we have $D \geq 2/\sqrt{\pi}$, by the Isodiametric inequality~\cite{evansGariepy15}, the leading constant in their bound is at least $72$ for any $K \in \mathcal{K}_2$. 
This result was strengthened by the authors~\cite{bsv2021} to $EH_{2,4}^K(n) \leq 12n^2 + o(n^2)$ for every $K \in \mathcal{K}_2$.
Here, we determine the leading constant in $EH_{2,4}^K(n)$ exactly.

\begin{theorem} 
\label{thm:4holes}
For every convex body $K \subseteq \mathbb{R}^2$ of unit area, we have
\[\lim_{n \to \infty} n^{-2}EH_{2,4}^K(n) = 10-\frac{2\pi^2}{3} \approx 3.42026. 
\]
\end{theorem}

Our computer experiments~\cite{supplemental_program} support this result and also equation~\eqref{eq-reitTemPlane}. We sampled random sets of $n=25 000$ points from a triangle, a square, and a disk, and for each shape the average number of 3-holes and 4-holes was around 
$2 n^2$ and $3.42 n^2$, respectively.
Moreover, 
the average number of 5-holes and 6-holes was around 
$3.4 n^2$ and $3 n^2$, respectively.
It may turn out that
the expected numbers of 4-holes and 5-holes coincide asymptotically.
Note that all these bounds for random point sets are worse than the ones obtained by the squared Horton sets of size $n$,
 which asymptotically contain 
$c_k n^2$ $k$-holes for $k \in\{3,\dots,6\}$, where
$c_3 \approx 1.6195$, 
$c_4 \approx 1.9396$,
$c_5 \approx 1.0206$, and
$c_6 \approx 0.2005$~\cite{BaranyValtr2004}.

For larger $k$-holes in the plane, we do not determine the value $\lim_{n \to \infty}n^{-2}EH_{2,k}^K(n)$ exactly, but we can show that it exists and does not depend on the convex body $K$.
We recall that this is not true in larger dimensions already for empty simplices by Corollary~\ref{cor:improvedLowerBound}.

\begin{theorem}
\label{thm:independence}
For every integer $k \geq 3$, there is a constant $C=C(k)$ such that, for every convex body $K \subseteq \mathbb{R}^2$ of unit area, we have
\[\lim_{n \to \infty}n^{-2}EH_{2,k}^K(n) = C.\]
\end{theorem}

The idea of the proof of Theorem~\ref{thm:largerHoles} is novel and the proof is presented in Section~\ref{sec-largerHoles}.
The proof of Theorem~\ref{thm:improvedLowerBound} is a combination of ideas used by Reitzner and Temesvari~\cite{ReitznerTemesvari2019} and of a result by Blaschke~\cite{Blaschke1917,Blaschke1923}.
The proof can be found in Section~\ref{sec:proof_thm_improvedLowerBound}, where we also prove Proposition~\ref{prop:improvedLowerBound2} using simple geometric arguments. 
Finally, Theorems~\ref{thm:4holes} and~\ref{thm:independence} are proved in Sections~\ref{sec:4holes} and~\ref{sec:independece}, respectively.
Both these proofs combine arguments used by Reitzner and Temesvari~\cite{ReitznerTemesvari2019} and by Valtr~\cite{Valtr1995}.

\paragraph{Open problems} 
We determined the leading constants $\lim_{n \to \infty }n^{-d}EH^K_{d,k}(n)$ exactly for small holes and, in particular, we showed that these limits exist in such cases.
However, we do not have any argument that would yield the existence of these limits for all values of $d$ and $k$.
It is thus an interesting open problem to determine whether $\lim_{n\to \infty }n^{-d}EH^K_{d,k}(n)$ exists for all positive integers $d$ and $k$ with $k \geq d+1$.
It follows from a result by Reitzner and Temesvari~\cite[Theorem~1.4]{ReitznerTemesvari2019} and from Theorem~\ref{thm:independence} that this limit exists if $k=d+1$ or if $k\geq 3$ and $d=2$, respectively.

We also pose the following problem asking for the value of the expected number of empty simplices in a tetrahedron.

\begin{problem}
\label{conj-simplex}
Let $K$ be a $3$-dimensional simplex of unit volume. Determine the leading constant $\lim_{n \to \infty} n^{-d} EH_{3,4}^K(n)$.
\end{problem}

It might also be interesting to determine $\lim_{n \to \infty}n^{-2}EH_{2,k}^K(n)$ exactly for as many values $k>4$ as possible.
Recall that, by Theorem~\ref{thm:independence}, the number $\lim_{n \to \infty}n^{-2}EH_{2,k}^K(n)$ is the same for all convex bodies $K \in\mathcal{K}_2$.

\section{Proof of Theorem~\ref{thm:largerHoles}}
\label{sec-largerHoles}

We show that, for all integers $d \geq 2$ and $k \geq d+1$, there are constants $C=C(d,k)>0$ and $n_0=n_0(d,k)$ such that, for every $n \geq n_0$ and every convex body $K \in \mathcal{K}_d$, we have $EH^K_{d,k}(n) \geq C \cdot n^d$.
In our proof, we make no serious attempt to optimize the constants.
We start by stating some auxiliary results and definitions.

We use $B^d$ to denote the $d$-dimensional ball of radius $1$ centered at the origin and $S^{d-1}$ to denote the sphere that is the boundary of $B^d$.
For $h \in [0,1]$, let $C(h) = \{(x_1,\dots,x_d) \in B^d \colon x_1 \geq 1-h\}$ be the \emph{spherical cap} of $S^{d-1}$ of height $h$.
For $x \in S^{d-1}$, we use $C(h,x)$ to denote the image of $C(h)$ under a rotation $r$ of $S^{d-1}$ around the origin such that $r(1,0,\dots,0)=x$.
We call the set $C(h,x)$ the \emph{spherical cap centered at~$x$}.
Note that $C(h,x)$ and $C(h)$ have the same height, volume, and surface area for every $x \in S^{d-1}$.

We first give a high-level overview of the proof.
For a given convex body $K \in \mathcal{K}_d$, we choose a sufficiently large $(d-1)$-dimensional ellipsoid $S''$ contained in $K$.
We then select $d$ points $q_1,\dots,q_d$ on $S''$ and we consider an $\varepsilon$-neighborhood $N_i$ of $q_i$ for every $i=1,\dots,d$ for a sufficiently small constant $\varepsilon = \varepsilon(d)>0$.
When we draw a set $S$ of $n$ points uniformly and independently at random from $K$, some $d$-tuple $p_1,\dots,p_d$ of the points from $S$ satisfies $p_i \in N_i$ with positive probability.
It then remains to show that the points $p_1,\dots,p_d$ are contained in a $k$-hole in $S$ with positive probability.
To do that, we use the choice of $S''$, which enables us to choose a $(d-1)$-dimensional ellipsoid $F \subseteq K$ containing the points $p_1,\dots,p_d$.
Moreover, since each $p_i$ is close to $S''$, the choice of $S''$ also guarantees the existence of a $d$-dimensional ellipsoid $G \subseteq K$ of volume $1/(n-d)$ such that $F \subseteq G$.
We then choose $k-d$ pairwise disjoint caps $C_1,\dots,C_{k-d}$ of $G$, each with volume at least $c$ for some constant $c=c(d,k)>0$, and we show that, with positive probability, there is a point $p_{i+d}$ of $S$ inside each cap $C_i$ and there is no point of $P$ inside $G \setminus (C_1\cup\cdots\cup C_{k-d})$.
The points $p_1,\dots,p_k$ then determine a $k$-hole in $S$ with positive probability.

Although the main idea of the proof of Theorem~\ref{thm:largerHoles} is relatively simple, there are many technical steps along the way and we need to prove several auxiliary results.
As the first such result, we use the following lower bound on the volume of spherical caps.

\begin{lemma}[\cite{miccVoul10}]
\label{lem-capVolume}
For every integer $d \geq 2$ and $h \in (0,1]$, we have
\[\lambda_d(C(h)) > \left(\sqrt{2h-h^2}\right)^{d-1} \cdot \frac{h}{2 d} \cdot \lambda_d(B^d).\]
\end{lemma}

The following lemma is used to estimate the surface area of $C(h)$.

\begin{lemma}
\label{lem-capArea}
For every $h \in (0,1/4]$, we have
\[\lambda_{d-1}(C(h) \cap S^{d-1}) \leq (4\sqrt{h})^{d-1}\lambda_{d-1}(S^{d-1}).\]
\end{lemma}
\begin{proof}
Let $\overline{C}$ be the convex hull of $C(h) \cup \{0\}$, where $0$ denotes the origin.
Then $\frac{\lambda_{d-1}(C(h) \cap S^{d-1})}{\lambda_{d-1}(S^{d-1})} = \frac{\lambda_d(\overline{C})}{\lambda_d(B^d)}$ and it thus suffices to find an upper bound on $\lambda_d(\overline{C})$.
The basis of $C(h)$ on $S^{d-1}$ is a $(d-1)$-dimensional ball of radius $\sqrt{2h-h^2}$.
Let $B$ be the $(d-1)$-dimensional ball with radius $\sqrt{2h-h^2}/(1-h)$ centered at $(1,0,\dots,0)$ lying in the hyperplane $x_1=1$.
The set $\overline{C}$ is then contained in the convex hull $\overline{B}$ of $B \cup \{0\}$.
Thus,
\[\lambda_d(\overline{C}) \leq \lambda_d(\overline{B}) = \frac{1}{d} \cdot \left(\frac{\sqrt{h(2-h)}}{1-h}\right)^{d-1} \cdot \lambda_{d-1}(B^{d-1}) \leq (4\sqrt{h})^{d-1} \lambda_d(B^d),\]
where we used $\sqrt{2-h}/(1-h) \leq 2$ for $h \leq 1/4$ together with a crude estimate $\lambda_{d-1}(B^{d-1}) \leq d2^{d-1}\lambda_d(B^d)$. 
We thus obtain $\lambda_{d-1}(C(h) \cap S^{d-1}) \leq (4\sqrt{h})^{d-1} \lambda_{d-1}(S^{d-1})$.
\end{proof}

Using this lemma, we show that, for every $k \in \mathbb{N}$, we can place $k$ pairwise disjoint spherical caps on $S^{d-1}$ so that each cap has height that depends only on $d$ and $k$.

\begin{lemma}
\label{lem-capsGreedy}
For integers $d \geq 2$, $k \geq d$, and $h \in (0,1/(64k^{2/(d-1)})]$, let $x_1\dots,x_d$ be points from~$S^{d-1}$.
Then, there are pairwise disjoint spherical caps $C(h,x_{d+1}),\dots,\allowbreak C(h,x_k)$ on $B^d$ such that $x_i \notin C(h,x_j)$ for all $i=1,\dots,d$ and $j=d+1,\dots,k$.
\end{lemma}
\begin{proof}
We proceed by induction on $k$.
The statement is trivial for $k=d$, so we assume that there are pairwise disjoint spherical caps $C(h,x_{d+1}),\dots,C(h,x_{k-1})$ for some $k \geq d+1$ and $h \in (0,1/(64k^{2/(d-1)})] \subseteq (0,1/(64(k-1)^{2/(d-1)})]$ such that $x_i \notin C(h,x_j)$ for all $i=1,\dots,d$ and $j=d+1,\dots,k-1$.
Let $U = S^{d-1} \cap (\bigcup_{i=1}^{k-1} C(4h,x_i))$.

By Lemma~\ref{lem-capArea}, we have 
\[\lambda_{d-1}(U) \leq (k-1) (8\sqrt{h})^{d-1} \lambda_{d-1}(S^{d-1}).\]
Since $h \leq 1/(64k^{2/(d-1)})$, we obtain $\lambda_{d-1}(U) < \lambda_{d-1}(S^{d-1})$ and thus there is a point $x_k \in S^{d-1} \setminus U$.

For every $h' \in [0,1]$ and $x \in S^{d-1}$, the basis of each spherical cap $C(h',x)$ on $S^{d-1}$ is a $(d-1)$-dimensional ball of radius $\sqrt{2h'-(h')^2}$.
Thus, the distance between $x$ and any of point of $C(h',x)$ is at most $\sqrt{2h'-(h')^2 + (h')^2} = \sqrt{2h'}$.
On the other hand, every point of $S^{d-1}$ at distance at most $\sqrt{2h'}$ from $x$ lies in $C(h',x)$.
Since $x_k \notin U$, the distance between $x_k$ and $x_i$ is larger than $\sqrt{2\cdot 4h}=2\sqrt{2h}$ for every $i=1,\dots,k-1$.
Also, any point $y \in C(h,x_k)$ is at distance larger than $2\sqrt{2h}-\sqrt{2h} = \sqrt{2h}$ from $x_i$ and thus $y \notin C(h,x_i)$.
Therefore, by the induction hypothesis, the spherical caps $C(h,x_{d+1}),\dots,C(h,x_k)$ are pairwise disjoint and $x_i \notin C(h,x_j)$ for all $i =1,\dots,d$ and $j=d+1,\dots,k$.
\end{proof}

For an integer $n>d$, a $d$-dimensional ellipsoid $D \subseteq K$ with $\lambda_d({\rm conv}(D))=1/(n-d)$, and points $q_1,\dots,q_d$ on the boundary of $D$, we use $E(D,q_1,\dots,q_d)$ to denote the event that, for the set $S$ of $n-d$ points drawn uniformly independently at random from $K$, there is a $k$-hole $\{p_1,\dots,p_k\}$ in $S \cup \{q_1,\dots,q_d\}$ satisfying $p_i = q_i$ for every $i=1,\dots,d$. 

\begin{lemma}
\label{lem-probCaps}
There is $n_0=n_0(d,k) \in \mathbb{N}$ such that for every integer $n \geq n_0$, every $d$-dimensional ellipsoid $D \subseteq K$ with $\lambda_d({\rm conv}(D))=1/(n-d)$, and any points $q_1,\dots,q_d$ on the boundary of $D$, there is a constant $C = C(d,k)>0$ such that the probability of the event $E(D,q_1,\dots,q_d)$ is at least $C$.
\end{lemma}
\begin{proof}
Let $f \colon \mathbb{R}^d \to \mathbb{R}^d$ be a one-to-one affine mapping such that $f(D) = B^d$.
We set $h=1/(64k^{2/(d-1)})$.
By Lemma~\ref{lem-capsGreedy}, there are pairwise disjoint spherical caps $C(h,f(q_{d+1})),\dots,\allowbreak C(h,f(q_k))$ on $B^d$ such that $f(q_i) \notin C(h,f(q_j))$ for all $i=1,\dots,d$ and $j=d+1,\dots,k$.
We let $C'_i = f^{-1}(C(h,f(q_i)))$ for every $i=d+1,\dots,k$.
We then fix the set $D' = f^{-1}(B^d \setminus (\bigcup_{i=d+1}^k C(h,f(q_i)))) = D \setminus (\bigcup_{i=d+1}^k C'_i)$. 

Let $S$ be the set of $n-d$ points drawn at random from $K$.
Let $E(D',0)$ be the event that no point from $S$ lies in the set $D'$.
By Lemma~\ref{lem-capVolume}, the volume of each spherical cap $C(h,f(q_i))$ is at least $c \cdot \lambda_d(B^d)$ for some constant $c=c(d,k)>0$.
Since $f$ preserves ratios of volumes, the volume of each set $C'_i$ is at least $c/(n-d)$.
It then follows from the fact that the spherical caps $C(h,f(q_{d+1})),\dots,C(h,f(q_k))$ are pairwise disjoint that the volume of $D'$ is at most $\frac{1}{n-d} - (k-d)\frac{c}{n-d} = \frac{1-c(k-d)}{n-d}$.
We may assume that $\frac{1-c(k-d)}{n-d}>0$, as otherwise $D' = \emptyset$ and $\Pr[E(D',0)]=1$.

Since the volume of the set $D'$ is at most $(1-c(k-d))/(n-d)$, we have
\[\Pr[E(D',0)] \geq \left(1-\frac{1-c(k-d)}{n-d}\right)^{n-d} \geq e^{-2+2c(k-d)} \geq e^{-2},\]
where we used the inequality $1-x \geq e^{-2x}$ for $x \in [0,1/2]$.
Note that indeed $(1-c(k-d))/(n-d) \in [0,1/2]$ if $n$ is sufficiently large with respect to $d$ and $k$.

We partition the set $S$ into $k-d$ sets $B_{d+1},\dots,B_k$, each of size at least $\lfloor (n-d)/(k-d)\rfloor \geq (n-d)/(k-d) -1 \geq 1$.
For each $i \in \{d+1,\dots,k\}$, let $E(D',i)$ be the event that no point from $B_i$ appears in~$C'_i$.
Then
\begin{align*}
\Pr\left[E(D',i) \mid E(D',0)\right] &= \left(\frac{1 - \lambda_d(D')-\lambda_d(C'_i)}{1-\lambda_d(D')}\right)^{|B_i|} = \left(1-\frac{\lambda_d(C'_i)}{1-\lambda_d(D')}\right)^{|B_i|} \\
&\leq \left(1-\lambda_d(C'_i)\right)^{|B_i|} \leq \left(1-\frac{c}{n-d}\right)^{\frac{n-d}{k-d}-1}  \leq 2e^{-c/(k-d)},
\end{align*}
since $\lambda_d(C'_i) \geq c/(n-d)$ 
and $1/(1-\frac{c}{n-d}) \le 2$ for $n$ sufficiently large.
Thus the event $\overline{E(D',i)}$ that there is a point from $B_i$ lying in $C'_i$ satisfies $\Pr[\overline{E(D',i)} \mid E(D',0)] \geq 1-2e^{-c/(k-d)}$.
Since the events $\overline{E(D',d+1)},\dots,\overline{E(D',k)}$ are mutually independent, the conditioned probability that there is a point from $B_i$ in $C'_i$ for every $i \in \{d+1,\dots,k\}$ satisfies
\[\Pr\left[\bigcap_{i=d+1}^k\overline{E(D',i)} \mid E(D',0)\right] \geq \left(1-2e^{-c/(k-d)}\right)^{k-d}.\]

We now show that if the event $E(D',0) \cap \left(\bigcap_{i=d+1}^k\overline{E(D',i)}\right)$ holds, then there is a $k$-hole $\{p_1,\dots,p_k\}$ in $S \cup\{q_1,\dots,q_d\}$ satisfying $p_i=q_i$ for every $i=1,\dots,d$ and $p_j \in C'_j$ for every $j=d+1,\dots,k$.
That is, we show $E(D,q_1,\dots,q_d) \supseteq E(D',0) \cap \left(\bigcap_{i=d+1}^k\overline{E(D',i)}\right)$.
Since there is a point from $B_i$ in $C'_i$ for every $i=d+1,\dots,k-d$, we can choose a point $p_i \in C'_i \cap S$ that is closest to $D'$ among all points from $C'_i \cap S$.
We also let $p_i=q_i$ for every $i=1,\dots,d$.
By the choice of the spherical caps $C(h,q_{d+1}),\dots,C(h,q_k)$, the points $p_1,\dots,p_k$ are in convex position.
Since each point $p_i$ is the closest one to $D'$, there is no point from $S \cap C'_i$ in the interior of ${\rm conv}(\{p_1,\dots,p_k\})$.
Since there is no point from $S$ in the set $D'$, we see that the points $p_1,\dots,p_k$ form a desired $k$-hole in $S \cup \{q_1,\dots,q_d\}$.

It thus remains to estimate the probability of the event $E(D',0) \cap \left( \bigcap_{i=d+1}^k\overline{E(D',i)}\right)$ from below.
We have
\begin{align*}\Pr\left[E(D',0) \cap \left( \bigcap_{i=d+1}^k\overline{E(D',i)}\right)\right] &= \Pr\left[\bigcap_{i=d+1}^k\overline{E(D',i)} \mid E(D',0)\right]\cdot \Pr\left[E(D',0)\right]\\
&\geq \left(1-2e^{-c/(k-d)}\right)^{k-d} \cdot \frac{1}{e^2} > 0,
\end{align*}
which finishes the proof of Lemma~\ref{lem-probCaps}.
\end{proof}

For $\varepsilon > 0$, a $d$-dimensional simplex with vertices $r_1,\dots,r_{d+1}$ is \emph{$\varepsilon$-regular}, if there exists a regular $d$-dimensional simplex with vertices $r'_1,\dots,r'_{d+1}$ such that $r_i$ is at distance at most $\varepsilon$ from $r'_i$ for every $i=1,\dots,d+1$.

The following auxiliary result states that, for every $\varepsilon$-regular simplex $R$ with a sufficiently small $\varepsilon>0$, there is a circumscribed ellipsoid of $R$ with roughly the same diameter. 

\begin{lemma}
\label{lem-almostRegular}
For an integer $d \geq 1$ and a real number $D >0$, let $\varepsilon \in (0,D/(4\sqrt{2d(d+1)}))$.
If $R$ is a $d$-dimensional $\varepsilon$-regular simplex in $\mathbb{R}^d$ with diameter $D$, then there is a number $c = c(d)$ and a circumscribed ellipsoid of $R$ with diameter at most $c D$.
\end{lemma}
\begin{proof}
Let $r_1,\dots,r_{d+1}$ be the vertices of $R$.
Since $R$ is $\varepsilon$-regular, there is a regular simplex $R'$ with vertices $r'_1,\dots,r'_{d+1}$ such that $r_i$ is at distance at most $\varepsilon$ from $r'_i$ for every $i=1,\dots,d+1$.
By the choice of $R$, the diameter of $R'$ is at least $D-2\varepsilon$ and most $D + 2\varepsilon$.
In particular, the side length $a$ of $R'$ satisfies $a \in [D-2\varepsilon,D+2\varepsilon] \subseteq [D/2,2D]$, as $\varepsilon \leq D/4$. 

The distance of the barycenter of $R'$ from its facets is $a\sqrt{\frac{d+1}{2d}}\cdot \frac{1}{d+1} = \frac{a}{\sqrt{2d(d+1)}}$. 
Thus $R'$ contains an inscribed $d$-dimensional ball $B$ of radius $\frac{a}{\sqrt{2d(d+1)}}$.
Since $\varepsilon < \frac{D}{4\sqrt{2d(d+1)}}$ and $a \geq D/2$, we have $\varepsilon < \frac{a}{2\sqrt{2d(d+1)}}$. 
On the other hand, since every vertex of $R'$ lies at distance $a\sqrt{\frac{d+1}{2d}}\cdot \frac{d}{d+1} = a\sqrt{\frac{d}{2(d+1)}}$ from the barycenter of $R'$, there is a circumscribed ball $B'$ of $R'$ with radius $a\sqrt{\frac{d}{2(d+1)}}$.
Note that $B '= dB$.

Since $R$ is an affine image of $R'$, there is a circumscribed ellipsoid $E$ of $R$ such that $R$ contains an inscribed copy of $\frac{1}{d}E$.
If $a_1 \leq \dots \leq a_d$ are the semi axes of $E$, then the volume $V$ of $E$ equals $\kappa_d (a_1\cdots a_d)$, where $\kappa_d$ is the volume of the $d$-dimensional unit ball.
Clearly, the volume of $R$ is at most $D^d/d!$.
Since $\frac{1}{d} E \subseteq R \subseteq E$, we have $Vd^{-d} \leq \lambda_d(R) \leq V$ and, in particular, $V \leq d^d\lambda_d(R)\leq d^d D^d/d!$.

Since $B \subseteq R'$ and $\varepsilon < \frac{a}{2\sqrt{2d(d+1)}}$, it follows that $R$ contains a ball of radius $r=\frac{a}{\sqrt{2d(d+1)}}-\varepsilon \geq \frac{a}{2\sqrt{2d(d+1)}}$ and thus $a_1,\dots,a_d \geq r$.
The radius $r$ is at least $\frac{D}{4\sqrt{2d(d+1)}}$, since $a \geq D/2$.
Since $\kappa_d (a_1\cdots a_d) = V \leq d^d D^d / d!$, we obtain 
\[a_d \leq \frac{d^dD^d}{d!\cdot \kappa_d \cdot a_1\cdots a_{d-1}} \leq \frac{d^d D^d}{d!\cdot \kappa_d \cdot r^{d-1}} \leq \frac{d^d (4\sqrt{2d(d+1)})^{d-1}}{d! \cdot \kappa_d} \cdot D,\]
where the second inequality follows from $a_1,\dots,a_{d-1} \geq r$ and the third one from $r \geq \frac{D}{4\sqrt{2d(d+1)}}$.
Since $a_1\leq \cdots \leq a_d$, the diameter of $E$ equals $2a_d \leq c D$, where $c=\frac{2d^d (4\sqrt{2d(d+1)})^{d-1}}{d! \cdot \kappa_d}$.
\end{proof}

We will also need the following classical result, known as \emph{John's lemma}~\cite{Joh48}.

\begin{lemma}[{\cite[Theorem 13.4.1]{Mat02}}]
\label{lem-LownerJohnElips}
For every $d$-dimensional convex body $K\subseteq\mathbb{R}^d$, there is a $d$-dimensional ellipsoid $E$ with the center $s$ satisfying
\[
E\subseteq K\subseteq s+d(E-s).
\]
\end{lemma}

Finally, we proceed with the proof of Theorem~\ref{thm:largerHoles}.
Let $K$ be a $d$-dimensional convex body of unit volume and let $k \geq d+1$ be an integer.
We set $\varepsilon>0$ to be sufficiently small with respect to $d$.

By Lemma~\ref{lem-LownerJohnElips}, there is a $d$-dimensional ellipsoid $E$ with the center $s$ that satisfies $E\subseteq K\subseteq s+d(E-s)$.
In particular, $\lambda_d(E) \geq \lambda_d(K)/d^d = d^{-d}$.
By applying a suitable affine volume preserving transformation, we can assume without loss of generality that $E$ is a $d$-dimensional ball $B$ of radius $r$ with the center at the origin.
Since $d^{-d} \leq \lambda_d(B) = \kappa_d r^d  \leq \pi^{d/2}r^d$, we have $r \geq 1/(d\sqrt{\pi})$.

We let $B'$ be the ball of radius $r' = r/2$ centered at the origin and $B''$ be the ball of radius $r''=(r'-3\varepsilon)/(2c+1)$ centered at the origin, where $c=c(d)$ is the constant from Lemma~\ref{lem-almostRegular}.
Since $3\varepsilon \leq 1/(4d\sqrt{\pi}) \leq r'/2 = r/4$ for $\varepsilon=\varepsilon(d)$ sufficiently small, we have $r'' \geq r/(8c+4)$ and the volume of the ball $B''$ is at least $(4d(2c+1))^{-d}$.
Moreover, since $r' = r'' + 3\varepsilon + 2cr''$, all points lying at distance at most $2cr''+3\varepsilon$ from an arbitrary point of $B''$ lie in $B'$.

Let $S''$ be a $(d-1)$-dimensional sphere obtained as the intersection of the boundary of~$B''$ with the hyperplane $x_d=0$  and let $q_1,\dots,q_d$ be points of a regular $(d-1)$-dimensional simplex $R'$ inscribed to $S''$.
We let $N_i$ be an $\varepsilon$-neighborhood of $q_i$ for each $i=1,\dots,p$.
Note that $\lambda_d(N_i) = \varepsilon^d\lambda_d(B^d)$ and $N_i \subseteq B'$ for every $i=1,\dots,d$. 

Let $p_1,\dots,p_n$ be the points drawn independently and uniformly at random from~$K$.
Assume that some distinct points $p_{i_1},\dots,p_{i_d}$ satisfy $p_{i_j} \in N_j$ for every $j=1,\dots,d$.
We will show that with positive probability that depends only on $d$ and $k$, there is a $k$-hole in $S$ containing the points $p_{i_1},\dots,p_{i_d}$.

Let $R$ be the $(d-1)$-dimensional simplex with vertices $p_{i_1},\dots,p_{i_d}$.
Since $p_{i_j} \in N_j$ for every $j=1,\dots,d$, the simplex $R$ is $\varepsilon$-regular.
Since $R' \subseteq B''$, the diameter of $R$ is at most $2r''+2\varepsilon$ and at least $2r''-2\varepsilon \geq r/(8c+4) -2\varepsilon \geq 1/(d\sqrt{\pi}(8c+4)) -2\varepsilon > 4(\sqrt{2(d-1)d})\varepsilon$ for $\varepsilon=\varepsilon(d)$ sufficiently small.
Thus, by Lemma~\ref{lem-almostRegular}, there is a circumscribed $(d-1)$-dimensional ellipsoid $F$ to $R$ with diameter at most $2c(r''+\varepsilon)$.
Since the points $p_{i_1},\dots,p_{i_d}$ lie at distance at most $\varepsilon$ from $B''$ and all points lying at distance at most $2cr''+3\varepsilon$ from an arbitrary point of $B''$ lie in $B'$, the triangle inequality gives $F \subseteq B'$ and, in particular, the diameter of $F$ is at most $2r'$.
Since $B' \subseteq K$, we have $F \subseteq K$ as well.

Since $F$ is circumscribed to the $\varepsilon$-regular simplex $R$, $F$ contains a $(d-1)$-dimensional ball of radius at least $t$ for some constant $t=t(d)>0$ and thus the volume of $F$ is at least $t^{d-1}$.
Let $G$ be a $d$-dimensional ellipsoid of $d$-dimensional volume $1/(n-d)$ such that $F$ is the intersection of $G$ with a hyperplane.
All points of $G$ lie at distance at most $d/((n-d) \cdot \lambda_{d-1}(F))$ from $F$, since otherwise $G$ contains the cone with basis $F$ and height larger than $d/((n-d) \cdot \lambda_{d-1}(F))$, which has $d$-dimensional volume larger than $\frac{1}{d}d\lambda_{d-1}(F) /((n-d) \cdot \lambda_{d-1}(F))=1/(n-d)$.
Since $\lambda_{d-1}(F) \geq t^{d-1}$, by choosing $n$ sufficiently large with respect to $d$, we have $d/((n-d) \cdot \lambda_{d-1}(F)) \leq r/2$.
Since all points at distance at most $r/2$ from $B'$ lie in $B$ and since $F \subseteq B'$, we see that the set of all points lying at distance at most $r/2$ from $F$ lie in $B$.
It follows that $G \subseteq B \subseteq K$.

By Lemma~\ref{lem-probCaps} applied to $G$ and points $p_{i_1},\dots,p_{i_d}$, the probability that there is a $k$-hole $\{p'_1\dots,p'_k\}$ in $S$ with $p'_j=p_{i_j}$ for $j=1,\dots,d$ is at least $C$ for some positive constant $C=C(d,k)$.

Let $E(i_1,\dots,i_d)$ be the event that $p_{i_j} \in N_j$ for every $j=1,\dots,d$.
The probability that $p_{i_j} \in N_j$ equals $\lambda_d(N_j) = (\varepsilon^d\lambda_d(B^d))$ and thus $\Pr[E(i_1\dots,i_d)] = (\varepsilon^d\lambda_d(B^d))^d$.
If $H(i_1,\dots,i_d)$ denotes the event that there is a $k$-hole in $S$ containing the points $p_{i_1},\dots,p_{i_d}$, then we have $\Pr[H(i_1,\dots,i_d) \mid E(i_1,\dots,i_d)] \geq C$.
Altogether, we obtain $\Pr[H(i_1,\dots,i_d)] \geq C(\varepsilon^d\lambda_d(B^d))^d$.

It follows that the expected number of $d$-tuples of points from $S$ that are contained in a $k$-hole in $S$ is at least $C(\varepsilon^d\lambda_d(B^d))^d\binom{n}{d}$.
Since every $k$-hole contains $\binom{k}{d}$ $d$-tuples of points, the expected number of $k$-holes in $S$ is at least \[C(\varepsilon^d\lambda_d(B^d))^d\binom{k}{d}^{-1}\binom{n}{d}.\]
This finishes the proof of Theorem~\ref{thm:largerHoles}.

\section{Proof of Theorem~\ref{thm:improvedLowerBound}}
\label{sec:proof_thm_improvedLowerBound}

Here we prove the lower bound $\lim_{n \to \infty} n^{-d}EH_{d,d+1}^K(n) \geq \frac{2}{(d-1)!p_{d-1}}$ for every integer $d \geq 2$ and every convex body $K \in \mathcal{K}_d$, where $p_{d-1}$ is the maximum probability that the convex polytope formed by the convex hull of $d+1$ points chosen uniformly and independently at random from a convex body in $\mathbb{R}^{d-1}$ has only $d$ vertices.
We also prove Corollary~\ref{cor:improvedLowerBound} and Proposition~\ref{prop:improvedLowerBound2} at the end of this section.

First, we state some definitions.
Recall that $\lambda_d$ denotes the Lebesgue measure on~$\RR^d$.
For a subspace~$E$ of $\RR^d$, we use $\lambda_E$ to denote the Lebesgue measure on~$E$ and, for a compact convex set $F$ with affine hull $E$, we use $\lambda_F$ to denote the restriction of $\lambda_E$ to~$F$.

Recall that the volume of the $d$-dimensional unit ball~$B^d$ in $\mathbb{R}^d$ is denoted by
\[
\kappa_d = \lambda_d(B^d) = \frac{\pi^{d/2}}{\Gamma(1+d/2)}.
\]
Similarly, we denote the surface area of the $d$-dimensional unit sphere~$S^{d-1}$ by
\[
\omega_d = \lambda_{d-1}(S^{d-1}) = d \cdot \kappa_d = \frac{2\pi^{d/2}}{\Gamma(d/2)}.
\]
We also let $\Delta_q(x_0,\dots,x_q)$ be the $q$-dimensional volume of the simplex $\conv(\{x_0,\ldots,x_q\})$.

For integers $d$ and $k$ with $1 \le k \le d$, we let $A(d,k)$ be the affine Grassmannian of $k$-dimensional affine subspaces of $\RR^d$ equipped with its standard topology.
The Grassmannian $A(d,k)$ is also equipped with the unique rigid motion invariant Haar measure $\mu_{k}$, normalized by $\mu_k(\{E \in A(d,k) : E \cap B^d \neq \emptyset \})=\kappa_{d-k}$. 
For more background, see Sections~1.3 and~5.1 in \cite{SchneiderWeil2008}.

The following result is a special case of the affine Blaschke--Petkantschin formula; see Theorem~7.2.7 in~\cite{SchneiderWeil2008}.

\begin{lemma}
\label{lem:integral}
For every integer $d \geq 2$, if $K'$ is a $d$-dimensional convex body, then
\[
 \lambda_d(K')^d
= (d-1)!\frac{\omega_d}{\omega_1} \int_{A(d,d-1)}\int_{(K' \cap E)^d}\Delta_{d-1} \; {\rm d}\lambda^d_E \mu_{d-1}({\rm d} E).
\]
\end{lemma}

Let $EV_d^{K'}$ denote the expected $d$-dimensional volume of the convex hull of $d+1$ points drawn uniformly and independently at random from a convex body $K' \subseteq \mathbb{R}^d$.
That is,
\[
EV_d^{K'} 
= \frac{1}{\lambda(K')^{d+1}}
\int_{K'} \ldots \int_{K'} \Delta_d(x_0,\ldots,x_d) 
\; \lambda({\rm d} x_0) \cdots \lambda({\rm d} x_d). 
\]

Let $K$ be a convex body in $\RR^d$ of unit $d$-dimensional volume. 
Since $\lambda_d(K)=1$, Lemma~\ref{lem:integral} gives
\begin{equation}
\label{eq-improvedLowerBound1}
1 = (d-1)!\frac{\omega_d}{\omega_1} \int_{A(d,d-1)}\int_{(K \cap E)^d}\Delta_{d-1} \; {\rm d}\lambda^d_E \mu_{d-1}({\rm d} E), 
\end{equation}
where the right side equals 
\begin{equation}
\label{eq-improvedLowerBound2}
(d-1)!\frac{\omega_d}{\omega_1} \int_{A(d,d-1)} \lambda_{d-1}(K \cap E)^{d} EV^{K \cap  E}_{d-1} \; \mu_{d-1}({\rm d} E)
\end{equation}
by the definition of $EV_{d-1}^{K \cap E}$.
Since
\begin{equation}
\label{eq-improvedLowerBoundEstimate}
\frac{(d+1)EV_{d-1}^{K \cap E}}{\lambda_{d-1}(K \cap E)}= p^{K \cap E}_{d-1} \leq p_{d-1}
\end{equation}
(see~\cite{klee69,schneider88}), we have $EV_{d-1}^{K \cap E}  \leq \frac{\lambda_{d-1}(K \cap E)p_{d-1}}{d+1}$. 
Using this estimate,
we bound~\eqref{eq-improvedLowerBound2} from above by
\[\frac{p_{d-1}}{d+1}(d-1)!\frac{\omega_d}{\omega_1} \int_{A(d,d-1)} \lambda_{d-1}(K \cap E)^{(d+1)} \; \mu_{d-1}({\rm d} E).\] 
By combining this estimate with~\eqref{eq-improvedLowerBound1}, we obtain
\begin{equation}
\label{eq:improvedLowerBound3}
\frac{\omega_1(d+1)}{\omega_d(d-1)!p_{d-1}} \leq \int_{A(d,d-1)} \lambda_{d-1}(K \cap E)^{d+1}\; \mu_{d-1}({\rm d} E).
\end{equation}

Now, we apply a result of Reitzner and Temesvari~\cite[Theorem~1.4]{ReitznerTemesvari2019}, which states
\[
\lim_{n \to \infty}
n^{-d} EH_{d,d+1}^K(n) 
= \frac{d \kappa_d}{(d+1)} \lambda_d(K)^{-d} \int_{A(d,d-1)} \lambda_{d-1}(K\cap E)^{d+1} \; \mu_{d-1}({\rm d} E).
\]
Together with~\eqref{eq:improvedLowerBound3}, $\lambda_d(K)=1$, $\omega_d = d \kappa_d$, and $\omega_1=2$, this yields the desired estimate
\[\lim_{n \to \infty}
n^{-d} EH_{d,d+1}^K(n)  \geq \frac{d \kappa_d}{(d+1)}  \cdot \frac{\omega_1(d+1)}{\omega_d(d-1)!p_{d-1}} = \frac{2}{(d-1)!p_{d-1}},\]
which completes the proof of Theorem~\ref{thm:improvedLowerBound}.
\qed

\medskip

We now prove Corollary~\ref{cor:improvedLowerBound} using
the following bounds on $p^K_2$ by Blaschke~\cite{Blaschke1917,Blaschke1923} and a result of Groemer~\cite{Groemer1973}.

\begin{theorem}[{\cite{Blaschke1917,Blaschke1923}}]
\label{thm:Blaschke}
For every 2-dimensional convex body $K \subseteq \mathbb{R}^2$, we have
\[
\frac{35}{12 \pi^2}
\le 
p^K_2
\le \frac{1}{3}.
\]
Moreover, the left inequality is tight if and only if $K$ is an ellipse, and the right inequality is tight if and only if $K$ is a triangle.
\end{theorem}

\begin{theorem}[{\cite{Groemer1973}}]
\label{thm:Groemer}
For every integer $d\geq 2$ and every $d$-dimensional convex body $K \subseteq \mathbb{R}^d$, we have
\[
EV^B_d \leq EV^K_d,
\]
where $B$ is the $d$-dimensional ball of the same volume as $K$.
Moreover, the inequality is tight if and only if $K$ is an ellipsoid.
\end{theorem}

\begin{proof}[Proof of Corollary~\ref{cor:improvedLowerBound}]
Let $\mathcal{C}_K$ be the space of all $(d-1)$-dimensional convex bodies that can be obtained as an intersection of a $d$-dimensional convex body $K$ with a hyperplane.
The probability $p_{d-1}^K$ is affine invariant and thus attains its minimum and maximum on~$\mathcal{C}_K$ for every $K$; we refer the reader to~\cite{macbeath51} and Chapter~1 in~\cite{schneider14}. 

Note that $\mathcal{C}_B$ contains only ellipsoids.
Since $\mathcal{C}_S$ does not contain ellipsoids, Theorem~\ref{thm:Groemer} gives 
\[
\min_{K \in \mathcal{C}_S} p^K_{d-1} > p^{B^{d-1}}_{d-1} = \max_{K \in \mathcal{C}_B} p^K_{d-1}.
\]
By estimating $EV^{K \cap E}_{d-1}$ in~\eqref{eq-improvedLowerBound2} from below with $(d+1)\lambda_{d-1}(K \cap E)\min_{K \in \mathcal{C}_S} p^K_{d-1}$ for $S$ and from above with $(d+1)\lambda_{d-1}(K \cap E)p^{B^{d-1}}_{d-1}$ for $B$, it then follows from the proof of Theorem~\ref{thm:improvedLowerBound} that
\[\lim_{n \to \infty} n^{-d}EH^S_{d,d+1}(n) < \lim_{n \to \infty} n^{-d}EH^B_{d,d+1}(n).\]

For the bounds in $\mathbb{R}^3$, the upper bound was proved by Reitzner and Temesvari~\cite{ReitznerTemesvari2019}. Thus, we only need to show the lower bound.
Theorem~\ref{thm:Blaschke} implies $p_2 = \frac{1}{3}$ and by plugging this identity into Theorem~\ref{thm:improvedLowerBound}, we obtain 
\[
\lim_{n \to \infty} n^{-3}EH_{3,4}^K \geq \frac{2}{2! \cdot \frac{1}{3}} = 3.
\]
\end{proof}

Finally, we show an upper bound on $p^K_d$ for every $d \geq 2$ and every convex body $K \subseteq \mathbb{R}^d$ with small diameter. 
That is, we prove Proposition~\ref{prop:improvedLowerBound2}.

\begin{proof}[Proof of Proposition~\ref{prop:improvedLowerBound2}]
We can assume without loss of generality that $f(K)=K$.
We prove the upper bound on $p^K_d$.
Since $\frac{(d+2)EV_d^K}{\lambda_d(K)}= p^{K}_d$ by~\eqref{eq-improvedLowerBoundEstimate}, 
it is sufficient to estimate the term $\frac{EV_d^K}{\lambda_d(K)}$, which is the expected $d$-dimensional volume $EV^K_d$ of the convex hull ${\rm conv}(\{p_1,\dots,p_{d+1}\})$ of $d+1$ points drawn uniformly from $K$ divided by the volume $\lambda_d(K)$ of $K$.
The volume of ${\rm conv}(\{p_1,\dots,\allowbreak p_{d+1}\})$ can be expressed as
$\frac{1}{d!}\prod_{i=2}^{d+1}h_i$, where $h_i$ is the distance of $p_i$ from the affine hull of the set $\{p_1,\dots,p_{i-1}\}$.
Since $K$ is convex, the distance $h_i$ is at most the diameter $D$ of $K$ for every $i \in \{2,\dots,d+1\}$ and thus
\[
\frac{EV^K_d}{\lambda_d(K)} \leq \frac{1}{d!\lambda_d(K)}\prod_{i=2}^{d+1}h_i \leq \frac{D^d}{d!\lambda_d(K)}.
\]
Altogether, we have $p_d \leq \frac{(d+2)D^d}{d!\lambda_d(K)}$.
\end{proof}

\section{Proof of Theorem~\ref{thm:4holes}}
\label{sec:4holes}
 
Let $K$ be a planar convex body of unit area.
Let $S$ be a set of $n$ points chosen uniformly and independently at random from~$K$.
We prove Theorem~\ref{thm:4holes} by showing that the expected number $E^K_{2,4}(n)$ of $4$-holes in $S$ satisfies 
\[\lim_{n \to \infty} n^{-2}E^K_{2,4}(n) = 10 - \frac{2\pi^2}{3}.\]

We distinguish two possible \emph{types} of $4$-holes and estimate their expected numbers separately.
In the first type, the pair of vertices of a 4-hole $H$ in $S$ at the largest distance between any two vertices of $H$ forms a diagonal of $H$.
In the other type, this pair of vertices forms an edge of $H$; see Figure~\ref{fig:4holes_bound_1}.

\begin{figure}[htb]
  \centering
    \includegraphics[page=1]{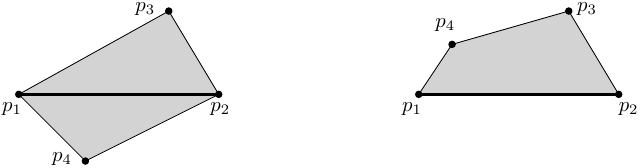}
  \caption{Two types of 4-holes.}
  \label{fig:4holes_bound_1}
\end{figure}

We now estimate the expected number of the $4$-holes of the first type.
For points $p_1,p_2,p_3,p_4 \in S$, we estimate the probability 
of the event $E_1(p_3,p_4 \mid p_1,p_2)$ that there is a 4-hole $H=\{p_1,\dots,p_4\}$ in $S$, where  the points $p_1$ and $p_2$ form a diagonal and determine the largest distance between any two points from $H$, conditioned on the fact that $p_1$ and $p_2$ are already placed in $K$. 
That is, the probability that, for fixed $p_1$ and $p_2$, the points $p_3$ and $p_4$ determine a $4$-hole of the first type with diameter $|p_1p_2|$.

Without loss of generality we can assume that $p_1=(0,0)$ 
and $p_2=(\ell,0)$ for some $\ell>0$,
as otherwise we apply a suitable isometry to~$S$.
Let $R$ be the set of points from $K \cap ([0,\ell] \times [-\frac{2}{\ell},\frac{2}{\ell}])$ 
that are at distance at most $\ell$ from $p_1$ and also from~$p_2$.
Note that the set $R$ is convex.
For every point $p \neq p_1,p_2$ of the $4$-hole $H$, we have $x(p_1)<x(p)<x(p_2)$, as otherwise $|p_1p_2|$ is not the diameter of $H$.
If $|y(p)|>\frac{2}{\ell}$, then the triangle $T$ spanned by $p_1$, $p_2$, and $p$ has area larger than $1$, which is impossible, since we have $T \subseteq K$ and $\lambda_2(K)=1$ by the convexity of~$K$.
Consequently, $p$ lies in $R$. 
For a real number $y \in [-\frac{2}{\ell},\frac{2}{\ell}]$, let $I_y$ be the line segment formed by points $r \in R $ with $y(r)=y$.
Note that $|I_y| \leq \ell$ for every $y$ and that $|I_0| = \ell$.

The event $E_1(p_3,p_4 \mid p_1,p_2)$ is partitioned into two events $E_1^{\myarrow[90]}(p_3,p_4 \mid p_1,p_2)$ and $E^{\myarrow[270]}_1(p_3,p_4 \mid p_1,p_2)$ distinguished by the position of $p_3$ and $p_4$ (either $p_3$ lies above $\overline{p_1p_2}$ and $p_4$ below or vice versa).
By symmetry, it suffices to estimate the probability of $E_1^{\myarrow[90]}(p_3,p_4 \mid p_1,p_2)$.

The event $E^{\myarrow[90]}_1(p_3,p_4 \mid p_1,p_2)$ happens 
if and only if
$p_3,p_4$ lie in $R$, both triangles
$p_1p_2p_3$ and $p_1p_2p_4$ are 3-holes in $S$, and $|p_3p_4|\leq|p_1p_2|$.
Note that $\lambda_2(p_1p_2p_3) + \lambda_2(p_1p_2p_4) = \frac{\ell}{2} (y(p_3)-y(p_4)) \le 1$ by the convexity of~$K$, so we get $-y(p_4) \le 2/\ell-y(p_3)$. 
Therefore, the probability $\Pr[E^{\myarrow[90]}_1(p_3,p_4 \mid p_1,p_2)]$ is at most
\begin{equation}
\label{eq-4HolesUpper}
\int_{y = 0}^{2/\ell} 
\int_{z = 0}^{2/\ell-y} 
|I_{y}| \cdot 
|I_{-z}| \cdot 
\Pr[\text{$p_1p_2p_3p_4$ is empty in }S \mid p_3 \in I_{y},p_4 \in I_{-z}] \;{\rm d}z {\rm d}y.
\end{equation}
Let $I^y_z$ be the set of points from $I_z$ that are at distance at most $\ell$ from any point of $I_y$ and observe that $I^0_0=I_0$.
Then,
the probability $\Pr[E^{\myarrow[90]}_1(p_3,p_4 \mid p_1,p_2)]$ is at least
\begin{equation}
\label{eq-4HolesLower}
\int_{y = 0}^{2/\ell} 
\int_{z = 0}^{2/\ell-y} 
|I_{y}| \cdot 
|I^y_{-z}| \cdot 
\Pr[\text{$p_1p_2p_3p_4$ is empty in }S \mid p_3 \in I_{y},p_4 \in I^y_{-z}] \;{\rm d}z {\rm d}y.
\end{equation}

Note that 
\[
\Pr[\text{$p_1p_2p_3p_4$ is empty in } S \mid p_3 \in I_{y},p_4 \in I_{-z}] = \left(1-\frac{ \ell \cdot (y + z)}{2}\right)^{n-4},
\]
as all the remaining $n-4$ points from $S$ must lie outside of $\conv(\{p_1,p_2,p_3,p_4\})$, which has area $\lambda_2(p_1p_2p_3) + \lambda_2(p_1p_2p_4) = \frac{\ell}{2} (y(p_3)-y(p_4))$.
The same expression holds for $\Pr[\text{$p_1p_2p_3p_4$ is empty in } S \mid p_3 \in I_{y},p_4 \in I^y_{-z}]$.
We now plug in this expression to~\eqref{eq-4HolesUpper} and~\eqref{eq-4HolesLower} and substitute $Y = yn$ and $Z=zn$.
Then,~\eqref{eq-4HolesUpper} becomes
\[
\frac{1}{n^2}\int_{Y = 0}^{2n/\ell} 
\int_{Z = 0}^{2n/\ell-Y} 
|I_{Y/n}| \cdot |I_{-Z/n}| \cdot 
\left(1-\frac{ \ell \cdot (Y + Z)}{2n}\right)^{n-4} \;{\rm d}Z {\rm d}Y
\]
while~\eqref{eq-4HolesLower} becomes
\[
\frac{1}{n^2}\int_{Y = 0}^{2n/\ell} 
\int_{Z = 0}^{2n/\ell-Y} 
|I_{Y/n}| \cdot |I^{Y/n}_{-Z/n}| \cdot 
\left(1-\frac{ \ell \cdot (Y + Z)}{2n}\right)^{n-4} \;{\rm d}Z {\rm d}Y.
\]
Note that we integrate over the set $\{(Y,Z) \in \mathbb{R}^2 \colon Y+Z \leq 2n/\ell, Y,Z \geq 0\}$, which becomes $\{(Y,Z) \in \mathbb{R}^2 \colon Y,Z \geq 0\}$ for $n$ going to infinity.
Also note that all sets $I_{Y/n}$, $I_{-Z/n}$, and $I^{Y/n}_{-Z/n}$ become $I_0$ as $n$ goes to infinity.
Thus, by the monotone convergence theorem, we have
\begin{align*}
    \lim_{n \to \infty} n^2 \Pr[E^{\myarrow[90]}_1(p_3,p_4 \mid p_1,p_2)] = 
    \int_{Y = 0}^{\infty} 
\int_{Z = 0}^{\infty} 
|I_{0}| \cdot |I_{0}| \cdot 
e^{-{ \ell \cdot (Y + Z)}/{2} } \;{\rm d}Z {\rm d}Y.
\end{align*}
Since $|I_{0}| = \ell$, we can simplify this expression as
\begin{align*}
    \lim_{n \to \infty} n^2 \Pr[E^{\myarrow[90]}_1(p_3,p_4 \mid p_1,p_2)] 
    &= 
     \int_{Y = 0}^{\infty}  \int_{Z = 0}^{\infty} \ell^2 e^{-{ \ell \cdot (Y + Z)}/{2} } \;{\rm d}Z {\rm d}Y 
     \\
    &= 
     { \int_{Y = 0}^{\infty} \ell e^{- \ell Y/{2} } \; {\rm d}Y }
     \cdot 
     { \int_{Z = 0}^{\infty} \ell e^{- \ell Z/{2} } \; {\rm d}Z }
      = 2 \cdot 2 = 4.
\end{align*}
Thus, the expected number of $4$-holes of the first type determined by $p_1$ and $p_2$ is $\binom{n}{2} \cdot 2 \cdot \frac{4}{n^2} = 4(1+o(1))$ for $n$ going to infinity.
It follows from the dominated convergence theorem that the overall expected number of $4$-holes of the first type is $4(1+o(1))\binom{n}{2} = 2(1+o(1))n^2$, 
independently on the location of the points $p_1$ and~$p_2$.

It remains to estimate the expected number of the 4-holes of the second type.
For points $p_1,p_2,p_3,p_4 \in S$, we estimate the probability of the event $E_2(p_3,p_4 \mid p_1,p_2)$ that, for fixed $p_1$ and $p_2$, there is a 4-hole $H=\{p_1,p_2,p_3,p_4\}$ in $S$, where  $p_1$ and $p_2$ form an edge of $\conv(H)$ and determine the largest distance between any two points from $H$ and $p_4$ is closer to the line $\overline{p_1p_2}$ than $p_3$. 
In particular, $p_1$ and $p_2$ determine a $4$-hole of the second type of diameter $|p_1p_2|$.

Again, we can assume without loss of generality that $p_1=(0,0)$ 
and $p_2=(\ell,0)$ for some $\ell>0$.
We consider the sets $R$ and $I_y$ for every $y\in [-\frac{2}{\ell},\frac{2}{\ell}]$ defined exactly as before.
Again, every point $p \neq p_1,p_2$ of the $4$-hole $H$ lies in $R$. 

The event $E_2(p_3,p_4 \mid p_1,p_2)$ is partitioned into four pairwise disjoint events $E_2^{\myarrow[135]}(p_3,p_4 \mid p_1,p_2)$, $E_2^{\myarrow[45]}(p_3,p_4 \mid p_1,p_2)$, $E_2^{\myarrow[-45]}(p_3,p_4 \mid p_1,p_2)$, and $E_2^{\myarrow[-135]}(p_3,p_4 \mid p_1,p_2)$ distinguished by the position of $p_3$ and $p_4$ with respect to the line $\overline{p_1p_2}$ (either both lie above the line or below) and by the position of $p_4$ with respect to $p_3$ ($p_4$ is either to the left or to the right of $p_3$).
First, we estimate the probability $\Pr[E_2^{\myarrow[135]}(p_3,p_4 \mid p_1,p_2)]$ of the event $E_2(p_3,p_4 \mid p_1,p_2)$ conditioned on the fact that $p_3$ and $p_4$ both lie above $\overline{p_1p_2}$ and $p_4$ lies to the left of~$p_3$; see Figure~\ref{fig:4holes_bound_3}.

\begin{figure}[htb]
  \centering
    \includegraphics[page=3]{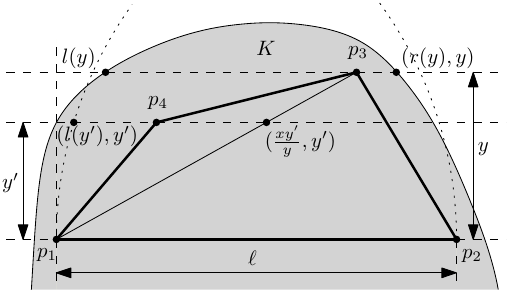}
  \caption{The case when the points $p_3=(x,y)$ and $p_4=(x',y')$ are both above $\overline{p_1p_2}$ and $p_4$ lies to the left of $p_3$.}
  \label{fig:4holes_bound_3}
\end{figure}

The event $E_2^{\myarrow[135]}(p_3,p_4 \mid p_1,p_2)$ happens if and only if $p_3,p_4$ lie in $R$ and both triangles $p_1p_2p_3$ and $p_1p_3p_4$ are 3-holes in $S$.
For $y \in [0,2/\ell]$, let $(l(y),y)$ and $(r(y),y)$ be the left and the right endpoint of the line segment $I_y$, respectively.
Note that $l(0) = 0$ and $r(0) = \ell$.
The point $p_3$ lies in $I_y$ for some $y \in [0,2/\ell]$ and the $x$-coordinate of $p_3$ then lies in the interval $[l(y),r(y)]$.
Since $p_4$ lies to the left of~$\overline{p_1p_3}$, the $x$-coordinate of $p_4$ lies in the interval $[l(y'),xy'/y']$, where $p_3=(x,y)$ and $y'$ is the $y$-coordinate of $p_4$.

Thus, we can express the probability $\Pr[E_2^{\myarrow[135]}(p_3,p_4 \mid p_1,p_2)]$ as
\[
\int_0^{2/\ell} 
\int_{l(y)}^{r(y)}
\int_0^y
\int_{l(y')}^{xy'/y}
P(p_1,p_2,p_3,p_4)\;
{\rm d}x' \,
{\rm d}y' \, 
{\rm d}x \,
{\rm d}y,
\]
where $P(p_1,p_2,p_3,p_4)$ is the probability that ${\rm conv}(\{p_1,p_2,p_3,p_4\})$ is empty in $S$, conditioned on the event $p_3=(x,y)$ and $p_4=(x',y')$.

\begin{figure}[htb]
  \centering
    \includegraphics[page=4]{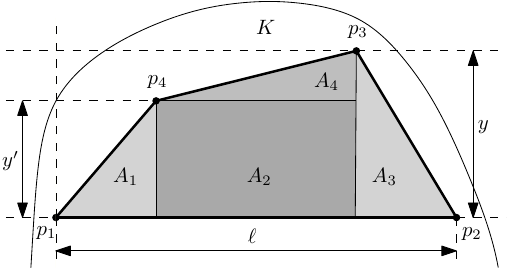}
  \caption{The partitioning of ${\rm conv}(\{p_1,p_2,p_3,p_4\})$.}
  \label{fig:4holes_bound_4}
\end{figure}

By partitioning the set ${\rm conv}(H) = {\rm conv}(\{p_1,p_2,p_3,p_4\})$ into four regions $A_1$, $A_2$, $A_3$, and $A_4$ as in Figure~\ref{fig:4holes_bound_4}, we can express the area $\lambda_2({\rm conv}(H))$ as $\lambda_2(A_1) + \cdots + \lambda_2(A_4)$, which equals
\[\frac{x'y'}{2} + (x-x')y'  + \frac{(\ell-x)y}{2} + \frac{(x-x')(y-y')}{2}.\]
Rewriting this expression, we obtain $\lambda_2(\conv(H)) = \frac{(\ell - x')y+xy'}{2}$.
Thus, we have
\[
P(p_1,p_2,p_3,p_4) = (1-\lambda_2(\conv(H)))^{n-4} 
= \left(1-\frac{(\ell - x')y+xy'}{2}\right)^{n-4},
\]
as $\lambda_2(K) = 1$ and all the remaining $n-4$ points from $S$ must lie in $K \setminus \conv(H)$.

Therefore, the probability $\Pr[E_2^{\myarrow[135]}(p_3,p_4 \mid p_1,p_2)]$ equals
\[\int_0^{2/\ell} 
\int_{l(y)}^{r(y)}
\int_0^y
\int_{l(y')}^{xy'/y}
\left(1-\frac{(\ell - x')y+xy'}{2}\right)^{n-4}\;
{\rm d}x' \,
{\rm d}y' \, 
{\rm d}x \,
{\rm d}y.\]
By Fubini's theorem, we can exchange the order of the integrals as
\[\int_0^{2/\ell} 
\int_0^y
\int_{l(y)}^{r(y)}
\int_{l(y')}^{xy'/y}
\left(1-\frac{(\ell - x')y+xy'}{2}\right)^{n-4}\;
{\rm d}x' \,
{\rm d}x \,
{\rm d}y' \, 
{\rm d}y.\]
We now substitute $Y = yn$ and $Y'=y'n$ and express $\Pr[E_2^{\myarrow[135]}(p_3,p_4 \mid p_1,p_2)]$ as
\[\frac{1}{n^2}
\int_0^{\frac{2n}{\ell}} 
\int_0^Y
\int_{l\left(\frac{Y}{n}\right)}^{r\left(\frac{Y}{n}\right)}
\int_{l\left(\frac{Y'}{n}\right)}^{\frac{xY'}{Y}}
\left(1-\frac{(\ell - x') Y+xY'}{2n}\right)^{n-4}
{\rm d}x' \,
{\rm d}x \,
{\rm d}Y' \, 
{\rm d}Y.\]
Since $l(0) = 0$ and $r(0)=\ell$, the monotone convergence theorem implies that \[\lim_{n \to \infty} n^2\Pr[E_2^{\myarrow[135]}(p_3,p_4 \mid p_1,p_2)]\]
equals
\[\int_0^{\infty} 
\int_0^Y
\int_0^\ell
\int_0^{xY'/Y}
e^{-((\ell - x')Y+xY')/2}\;
{\rm d}x' \,
{\rm d}x \,
{\rm d}Y' \, 
{\rm d}Y,\]
which can be rewritten as
\[\int_0^{\infty} 
e^{-\ell Y/2}
\int_0^Y 
\int_0^\ell
e^{-xY'/2}
\int_0^{xY'/Y}
e^{x'Y/2}\;
{\rm d}x' \,
{\rm d}x \,
{\rm d}Y' \, 
{\rm d}Y.\]
Computing the innermost integral, we obtain
\[\lim_{n \to \infty} n^2\Pr[E_2^{\myarrow[135]}(p_3,p_4 \mid p_1,p_2)] = 
2\int_0^{\infty} 
\frac{e^{-\ell Y/2}}{Y}
\int_0^Y 
\int_0^\ell
\left(1-e^{-xY'/2}\right)\;
{\rm d}x \,
{\rm d}Y' \, 
{\rm d}Y,\]
which can be further rewritten as
\[
2\int_0^{\infty} 
\frac{e^{-\ell Y/2}}{Y}
\int_0^Y 
\left({\ell}-\frac{1-e^{-\ell Y'/2}}{Y'/2}\right)\;
{\rm d}Y' \, 
{\rm d}Y
=
4\int_0^{\infty} 
\frac{e^{-z}}{z}
\int_0^z 
\left(1-\frac{1-e^{-t}}{t}\right)\;
{\rm d}t \, 
{\rm d}z
\]
by substituting $\frac{\ell}{2}Y=z$ and $\frac{\ell}{2}Y'=t$.
We split the integral into two parts. 
The first part equals
\[
4\int_0^{\infty} 
\frac{e^{-z}}{z}
\int_0^z 
1\;
{\rm d}t \, 
{\rm d}z
=
4\int_0^{\infty} 
e^{-z}
\;
{\rm d}z = 4.
\]

Using the
modified exponential integral function
(see Chapter~5.1 in \cite{abramowitz1965handbook})
\[
\int_0^z 
\frac{1-e^{-t}}{t}\;
{\rm d}t
=
\sum_{k=1}^\infty
\frac{(-1)^{k+1} \cdot z^k}
{k \cdot k!},
\]
we can write the second part as
\[
4\int_0^{\infty} 
\frac{e^{-z}}{z}
\int_0^z 
\frac{1-e^{-t}}{t}\;
{\rm d}t \, 
{\rm d}z
=
4\int_0^{\infty} 
\sum_{k=1}^\infty
\frac{(-1)^{k+1}}{k \cdot k!}
\cdot
e^{-z} z^{k-1}
\;
{\rm d}z
.
\]
The integral of the absolute value is finite, because 
\begin{align*}
4\int_0^{\infty} 
\sum_{k=1}^\infty
\frac{1}{k \cdot k!}
\cdot
e^{-z} z^{k-1}
\;
{\rm d}z
&=
4\sum_{k=1}^\infty
\int_0^{\infty} 
\frac{1}{k \cdot k!}
\cdot 
 e^{-z}  z^{k-1}
\;
{\rm d}z \\
&=
4\sum_{k=1}^\infty
\frac{1}{k \cdot k!}
\int_0^{\infty} 
 e^{-z}  z^{k-1}
\;
{\rm d}z
=
4\sum_{k=1}^\infty
\frac{1}{k^2}
=
\frac{2\pi^2}{3},
\end{align*}
where the first equality follows from Tonelli's theorem,
the second equality from the linearity of the integral,
and the third equality uses $\int_0^{\infty}  e^{-z}  z^{k-1} {\rm d}z = \Gamma(k) = (k-1)!$ for $k \in \mathbb{N}$.
Therefore, we can apply Fubini's theorem
and obtain
\[
4\int_0^{\infty} 
\sum_{k=1}^\infty
\frac{(-1)^{k+1}}{k \cdot k!}
\cdot
e^{-z} z^{k-1}
\;
{\rm d}z
=
4
\sum_{k=1}^\infty
\int_0^{\infty} 
\frac{(-1)^{k+1}}{k \cdot k!}
 e^{-z}  z^{k-1}
\;
{\rm d}z.
\]
Again, since $\int_0^{\infty}  e^{-z}  z^{k-1} {\rm d}z = \Gamma(k) = (k-1)!$ for $k \in \mathbb{N}$,
we have
\[
4
\sum_{k=1}^\infty
\frac{(-1)^{k+1}}{k \cdot k!}
\int_0^{\infty} 
 e^{-z}  z^{k-1}
\;
{\rm d}z
=
4
\sum_{k=1}^\infty
\frac{(-1)^{k+1}}{k^2}
=
\frac{\pi^2}{3}.
\]
Altogether, we have
\[\lim_{n \to \infty} n^2\Pr[E_2^{\myarrow[135]}(p_3,p_4 \mid p_1,p_2)] = 4 - \frac{\pi^2}{3}.\]

The probability of each of the remaining three pairwise disjoint events 
$E_2^{\myarrow[45]}(p_3,p_4 \mid p_1,p_2)$, 
$E_2^{\myarrow[-45]}(p_3,p_4 \mid p_1,p_2)$, and
$E_2^{\myarrow[-135]}(p_3,p_4 \mid p_1,p_2)$ 
can be computed in an analogous way and all of them are equal.
We then obtain 
\begin{align*}
\Pr[E_2(p_3,p_4 \mid p_1,p_2)] &= 
\Pr[E_2^{\myarrow[135]}(p_3,p_4 \mid p_1,p_2)] + \Pr[E_2^{\myarrow[45]}(p_3,p_4 \mid p_1,p_2)] \\
&+ \Pr[E_2^{\myarrow[-45]}(p_3,p_4 \mid p_1,p_2)]+ \Pr[E_2^{\myarrow[-135]}(p_3,p_4 \mid p_1,p_2)] 
\\
&= 4 \Pr[E_2^{\myarrow[135]}(p_3,p_4 \mid p_1,p_2)].
\end{align*}

Thus, the expected number of $4$-holes of the second type determined by $p_1$ and $p_2$ is $n^2 \left(\frac{16}{n^2} -  \frac{4\pi^2}{3n^2}\right) = 16 - \frac{4\pi^2}{3}$ for $n$ going to infinity.
It follows from the dominated convergence theorem that the overall expected number of $4$-holes of the second type is $(1+o(1))\left(16-\frac{4\pi^2}{3}\right)\binom{n}{2} = (1+o(1))\left(8- \frac{2\pi^2}{3}\right)n^2$.
Adding the expected number of the $4$-holes of the first type, we see that 
\[\lim_{n\to \infty}n^{-2}EH_{2,4}^K(n) = 2 + \left(8 - \frac{2\pi^2}{3}\right) = 10 - \frac{2\pi^2}{3},\]
which completes the proof of Theorem~\ref{thm:4holes}.

\section{Proof of Theorem~\ref{thm:independence}}
\label{sec:independece}

We extend the methods used in Section~\ref{sec:4holes} to larger holes in random planar point sets and we show that, for every fixed integer $k \geq 3$, the limit $\lim_{n \to \infty}n^{-2}EH_{2,k}^K(n)$ exists and does not depend on the convex body $K \in \mathcal{K}_2$.

Let $K$ be a convex body of unit area in the plane and let $k \geq 3$ be an integer.
We use $S$ to denote a set of $n$ points chosen uniformly and independently at random from~$K$.

We distinguish several types of $k$-holes in $S$ and we compute the expected number for each type separately.
Let $H$ be a $k$-hole in $S$ and let $p_1,\dots,p_k$ be the points of $H$ such that $|p_1p_k|$ is the diameter of ${\rm conv}(H)$.
We can assume without loss of generality that $p_1=(0,0)$ and $p_k=(\ell,0)$ for some $\ell>0$.
In particular, all the remaining points $p_2,p_3,\dots,p_{k-1}$ lie between $p_1$ and $p_k$, since $|p_1p_k|$ is the diameter of $H$.
We then label the points $p_1,\dots,p_k$ so that they are ordered by their increasing $x$-coordinates.
For $T \subseteq \{2,3,\dots,k-1\}$, we use $\overline{T}$ to denote the set $\{2,3,\dots,k-1\}\setminus T$.
For $a \in \overline{T}$ and $b \in T$, we say that $H$ is of \emph{type} $(T,a,b)$ if all points $p_i$ with $i \in T$ lie above the line $\overline{p_1p_k}$, all points from $\overline{T}$ lie below the line $\overline{p_1p_k}$,  
the point $p_b$ has the largest $y$-coordinate among the points in $T$ and
$p_a$ has the smallest $y$-coordinate among the points in $\overline{T}$; see Figure~\ref{fig:kHolesbound}.
If $T$ or $\overline{T}$ is empty, then we let $a$ or $b$ be equal to $\emptyset$, respectively.
Observe that there are exactly 
\begin{align*}
2(k-2) + \sum_{\substack{T \subseteq \{2,3,,\dots,k-1\}\\ T,\overline{T} \neq \emptyset}} |T|\cdot |\overline{T}| &= 2(k-2) + \sum_{i = 1}^{k-3} \binom{k-2}{i}i(k-2-i) \\
&= 2k-4 + 2^{k-4}(k^2-5k+6)
\end{align*}
types of $k$-holes in $S$.

\begin{figure}[htb]
  \centering
    \includegraphics{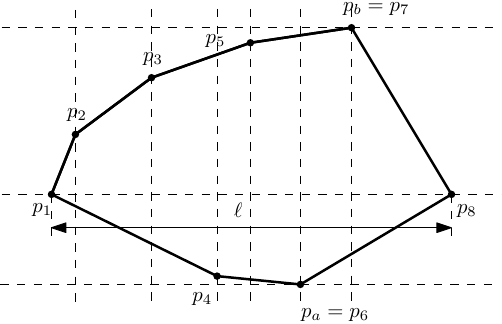}
  \caption{An example of an $8$-hole $\{p_1,\dots,p_8\}$ in $S$ of the type $(\{2,3,5,7\},6,7)$ with the diameter $\ell = |p_1p_8|$.}
  \label{fig:kHolesbound}
\end{figure}

Let $H$ be a $k$-hole of a type $(T,a,b)$ and let $T^{\myarrow[135]}$ be the set of points $p_i$ with $i \in T$ and $i < b$.
Then, since the points $p_1,\dots,p_k$ are in convex position, the points $p_i$ with $i \in T^{\myarrow[135]} \cup \{1,b\}$ have increasing $y$-coordinates.
Similarly, we define $T^{\myarrow[45]}$ to be the set of points $p_i$ with $i \in T$ and $i > b$ and $T^{\myarrow[-135]}$ and $T^{\myarrow[-45]}$ to be the sets of points $p_i$ with $i \in \overline{T}$ and with $i < a$ and $i>a$, respectively.
The convex position of the points $p_1,\dots,p_k$ then implies that the points $p_i$ with $i \in T^{\myarrow[45]} \cup \{b,k\}$ and $T^{\myarrow[-135]} \cup \{1,a\}$ have decreasing $y$-coordinates while the points $p_i$ with $T^{\myarrow[-45]} \cup \{a,k\}$ have increasing $y$-coordinates.

Let $(T,a,b)$ be a fixed type of $k$-holes in $S$.
For any two points $p_1,p_k \in K$, we estimate the probability of the event $E_{T,a,b}(p_1,\dots,p_k \mid p_1,p_k)$ that, for fixed $p_1$ and $p_k$, there is a $k$-hole $H=\{p_1,\dots,p_k\}$ in $S$ of the type $(T,a,b)$ with the diameter $|p_1p_k|$. 

We use $R$ to denote the set of points from $K \cap ([0,\ell] \times [-\frac{2}{\ell},\frac{2}{\ell}])$ that are at distance at most $\ell$ from $p_1$ and also from $p_k$.
Similarly as in Section~\ref{sec:4holes}, note that $R$ is convex and that every point $p\neq p_1,p_k$ of the $k$-hole $H$ lies in $R$.
For $y \in [-\frac{2}{\ell},\frac{2}{\ell}]$, let $I_y$ be the line segment formed by points $r \in R$ with $y(r) = y$.
We use $(l(y),y)$ and $(r(y),y)$ to denote the left and the right endpoint of $I_y$, respectively.
Note that $l(0) = 0$ and $r(0) = \ell$.

The point $p_a$ lies in $I_y$ for some $y \in [-\frac{2}{\ell},0]$ and $p_b$ in $I_{y'}$ for some $y' \in [0,\frac{2}{\ell}]$.
The $x$-coordinates of $p_a$ and $p_b$ then lie in the intervals $[l(y),r(y)]$ and $[l(y'),r(y')]$, respectively.

Let $p_{i_1},\dots,p_{i_m}$ be the points from $T^{\myarrow[135]} \cup \{p_1,p_b\}$ ordered by increasing $x$-coordinates.
We assume $p_{i_j}=(x_{i_j},y_{i_j})$ for every $j = 1,\dots,m$.
Since the points from $T^{\myarrow[135]} \cup \{p_1,p_b\}$ are in convex position with increasing $x$- and $y$-coordinates, we have, if $m \geq 3$,
\begin{equation}
\label{eq-coordinates1}
0\leq y_{i_{m-1}} \leq y_{i_m} \;\;\;\text{ and }\;\;\; l(y_{i_{m-1}}) \leq x_{i_{m-1}} \leq x_{i_m}\frac{y_{i_{m-1}}}{y_{i_m}}
\end{equation}
as $p_{i_{m-1}} \in I_{y_{m-1}}$ has to lie above the line $\overline{p_1p_{i_{m}}}$.
Moreover, for every $j = m-2,\dots,3,2$, we have 
\begin{equation}
\label{eq-coordinates2}
0 \leq y_{i_j} \leq y_{i_{j+1}}
\end{equation}
and
\begin{equation}
\label{eq-coordinates3}
\max \left\{l(y_{i_j}),x_{i_{j+1}} + \frac{(x_{i_{j+2}}-x_{i_{j+1}})(y_{i_j}-y_{i_{j+1}})}{y_{i_{j+2}}-y_{i_{j+1}}}\right\} \leq x_{i_j} \leq x_{i_{j+1}}\frac{y_{i_j}}{y_{i_{j+1}}}
\end{equation}
as $p_{i_j} \in I_{y_j}$ has to lie below the line $\overline{p_{i_{j+1}}p_{i_{j+2}}}$ and above the line $\overline{p_1p_{i_{j+1}}}$.
Let $J^{\myarrow[135]}_{T,a,b}$ be the set of vectors $(x_2,y_2,x_3,y_3,\dots,x_{k-1},y_{k-1})$ such that 
$0 \leq x_2 \leq \dots \leq x_{k-1} \leq \ell$ and the points $(x_{i_1},y_{i_1}),\dots,(x_{i_{m-1}},y_{i_{m-1}})$  
satisfy the conditions~\eqref{eq-coordinates1},~\eqref{eq-coordinates2} and~\eqref{eq-coordinates3}.
We analogously define the sets $J^{\myarrow[45]}_{T,a,b}$, $J^{\myarrow[-45]}_{T,a,b}$, and $J^{\myarrow[-135]}_{T,a,b}$ for the sets $T^{\myarrow[45]}$, $T^{\myarrow[-45]}$, and $T^{\myarrow[-135]}$, respectively.

We let $A^{\myarrow[135]}_{T,a,b}$ be the area of the convex hull of the points from $T^{\myarrow[135]} \cup \{p_1,p_b,(x(p_b),0)\}$; see Figure~\ref{fig:kHolesboundArea}.
By considering the partitioning into triangles and rectangles as in Figure~\ref{fig:kHolesboundArea}, the area $A^{\myarrow[135]}_{T,a,b}$ can be expressed as
\begin{align}
\begin{split}
\label{eq-area}
\sum_{j=2}^m \frac{(x_{i_j}  - x_{i_{j-1}})(y_{i_j} - y_{i_{j-1}})}{2} 
+ \sum_{j=2}^{m-1} (x_{i_m} - x_{i_j})(y_{i_j} - y_{i_{j-1}}).
\end{split}
\end{align}
We also let $A^{\myarrow[45]}_{T,a,b}$, $A^{\myarrow[-45]}_{T,a,b}$, and $A^{\myarrow[-135]}_{T,a,b}$ be the corresponding areas and we note that they can be evaluated using an analogous approach.
Observe that $\lambda_2({\rm conv}(H)) = A^{\myarrow[135]}_{T,a,b} + A^{\myarrow[45]}_{T,a,b} + A^{\myarrow[-45]}_{T,a,b} + A^{\myarrow[-135]}_{T,a,b}$.

\begin{figure}[htb]
  \centering
    \includegraphics{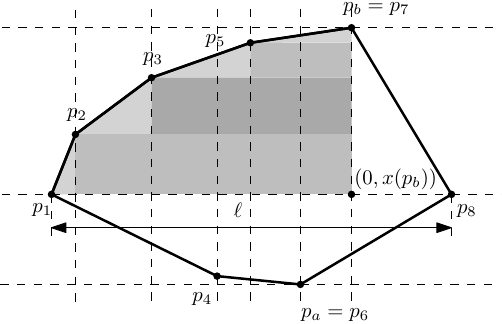}
  \caption{The area $A^{\myarrow[135]}_{T,a,b}$ of the convex hull of the points from $\{p_2,p_3,p_5\}\cup\{p_1,p_b,(x(p_b),0)\}$ is denoted by four grey triangles and three grey rectangles.}
  \label{fig:kHolesboundArea}
\end{figure}

The event $E_{T,a,b}(p_1,\dots,p_k \mid p_1,p_k)$ happens if all points from $S \setminus \{p_1,\dots,p_k\}$ lie in $K \setminus {\rm conv}(H)$.
Since $\lambda_2(K)=1$, we can express the probability $\Pr[E_{T,a,b}(p_1,\dots,p_k \mid p_1,p_k)]$ in the case $a<b$ as
\[
\int_0^{2/\ell}
\int_{l(y_b)}^{r(y_b)}
\int_{-2/\ell+y_b}^0
\int_{l(y_a)}^{\min\{x_b,r(y_a)\}}
\int_{J_{T,a,b}}
(1-\lambda_2({\rm conv}(H))^{n-k}\;
{\rm d}{\bf z} \,
{\rm d}x_a \,
{\rm d}y_a \,
{\rm d}x_b \, 
{\rm d}y_b
,\]
where $p_i=(x_i,y_i)$ for every $i=1,\dots,k$ and $J_{T,a,b}$ is the set of vectors $(x_2,y_2,\dots,x_{k-1},y_{k-1})$ from 
$J^{\myarrow[135]}_{T,a,b} \cap J^{\myarrow[45]}_{T,a,b} \cap J^{\myarrow[-45]}_{T,a,b} \cap J^{\myarrow[-135]}_{T,a,b}$ where any two points $(x_i,y_i)$ and $(x_j,y_j)$ are at distance at most $\ell$. 
In the case $a > b$ we can write $\Pr[E_{T,a,b}(p_1,\dots,p_k \mid p_1,p_k)]$ as
\[
\int_0^{2/\ell}
\int_{l(y_b)}^{r(y_b)}
\int_{-2/\ell+y_b}^0
\int_{\max\{x_b,l(y_a)\}}^{r(y_a)}
\int_{J_{T,a,b}}
(1-\lambda_2({\rm conv}(H))^{n-k}\;
{\rm d}{\bf z} \,
{\rm d}x_a \,
{\rm d}y_a \,
{\rm d}x_b \, 
{\rm d}y_b
.\]
In the case $a=\emptyset$, $\Pr[E_{T,a,b}(p_1,\dots,p_k \mid p_1,p_k)]$ equals
\[
\int_0^{2/\ell}
\int_{l(y_b)}^{r(y_b)}
\int_{J_{T,\emptyset,b}}
(1-\lambda_2({\rm conv}(H))^{n-k}\;
{\rm d}{\bf z} \,
{\rm d}x_b \, 
{\rm d}y_b
\]
and, similarly, if $b=\emptyset$, we have
\[
\int_{-2/\ell}^0
\int_{l(y_a)}^{r(y_a)}
\int_{J_{T,a,\emptyset}}
(1-\lambda_2({\rm conv}(H))^{n-k}\;
{\rm d}{\bf z} \,
{\rm d}x_a \,
{\rm d}y_a
.\]

If we substitute $Y_i = ny_i$ for every $i = 2,3,\dots,k-1$ and we let $n$ go to infinity, then conditions~\eqref{eq-coordinates1},~\eqref{eq-coordinates2}, and~\eqref{eq-coordinates3} become 
\[
0\leq Y_{i_{m-1}} \leq Y_{i_m} \;\;\;\text{ and }\;\;\; 0 \leq x_{i_{m-1}} \leq x_{i_m}\frac{Y_{i_{m-1}}}{Y_{i_m}},
\]
\[
0 \leq Y_{i_j} \leq Y_{i_{j+1}},
\]
and
\[
\max\left\{0,x_{i_{j+1}} + \frac{(x_{i_{j+2}}-x_{i_{j+1}})(Y_{i_j}-Y_{i_{j+1}})}{Y_{i_{j+2}}-Y_{i_{j+1}}}\right\} \leq x_{i_j} \leq x_{i_{j+1}}\frac{Y_{i_j}}{Y_{i_{j+1}}},
\]
respectively, since $l(0)=0$.
In particular, none of the new conditions depends on $K$.
Let $G^{\myarrow[135]}_{T,a,b}$ be the set of vectors $(x_2,Y_2,x_3,Y_3,\dots,x_{k-1},Y_{k-1})$, where $0\leq x_2 \leq \cdots \leq x_{k-1} \leq \ell$ and the points $(x_{i_1},Y_{i_1}),\dots,(x_{i_{m-1}},Y_{i_{m-1}})$ satisfy these new conditions.
We analogously define the sets $G^{\myarrow[45]}_{T,a,b}$, $G^{\myarrow[-45]}_{T,a,b}$, and $G^{\myarrow[-135]}_{T,a,b}$.

By substituting $Y_i = ny_i$ for every $i = 2,3,\dots,k-1$ and letting $n$ go to infinity, the expression~\eqref{eq-area} for the area $A^{\myarrow[135]}_{T,a,b}$ becomes
\[\sum_{j=2}^m \frac{(x_{i_j} - x_{i_{j-1}})(Y_{i_j} - Y_{i_{j-1}})}{2n} + \sum_{j=2}^{m-1} \frac{(x_{i_m}-x_{i_j})(Y_{i_j} - Y_{i_{j-1}})}{n} = \frac{B^{\myarrow[45]}_{T,a,b}}{n},
\]
where $B^{\myarrow[45]}_{T,a,b}$ is independent of $K$.
Similarly, the expressions for the areas $A^{\myarrow[45]}_{T,a,b}$, $A^{\myarrow[-45]}_{T,a,b}$, $A^{\myarrow[-135]}_{T,a,b}$ become $\frac{B^{\myarrow[45]}_{T,a,b}}{n}$, $\frac{B^{\myarrow[-45]}_{T,a,b}}{n}$, $\frac{B^{\myarrow[-135]}_{T,a,b}}{n}$, respectively, where $B^{\myarrow[45]}_{T,a,b}$, $B^{\myarrow[-45]}_{T,a,b}$, $B^{\myarrow[-135]}_{T,a,b}$ are independent of $K$.
Since $\lambda_2({\rm conv}(H)) = A^{\myarrow[135]}_{T,a,b} + A^{\myarrow[45]}_{T,a,b} + A^{\myarrow[-45]}_{T,a,b} + A^{\myarrow[-135]}_{T,a,b}$, we have 
$\lambda_2({\rm conv}(H)) = (B^{\myarrow[135]}_{T,a,b} + B^{\myarrow[45]}_{T,a,b} + B^{\myarrow[-45]}_{T,a,b} + B^{\myarrow[-135]}_{T,a,b})/n$.

Thus by substituting $Y_i = ny_i$ for every $i = 2,3,\dots,k-1$ and using the monotone convergence theorem with the facts $l(0)=0$ and $r(0)=\ell$, 
we can express $\lim_{n \to \infty} n^{k-2}\Pr[E_{T,a,b}(p_1,\dots,p_k \mid p_1,p_k)]$ in the case $a<b$ as
\[
\int_0^\infty
\int_0^\ell
\int_{-\infty}^0
\int_0^{x_b}
\int_{G_{T,a,b}}
e^{-B^{\myarrow[135]}_{T,a,b} - B^{\myarrow[45]}_{T,a,b} - B^{\myarrow[-45]}_{T,a,b} - B^{\myarrow[-135]}_{T,a,b}}\;
{\rm d}{\bf Z} \,
{\rm d}x_a \,
{\rm d}Y_a \,
{\rm d}x_b \, 
{\rm d}Y_b
,\]
where $p_i=(x_i,Y_i/n)$ for every $i=1,\dots,k$ and $G_{T,a,b}$ is the set of vectors $(x_2,Y_2,\dots,\allowbreak x_{k-1},Y_{k-1})$ from $G^{\myarrow[135]}_{T,a,b} \cap G^{\myarrow[45]}_{T,a,b} \cap G^{\myarrow[-45]}_{T,a,b} \cap G^{\myarrow[-135]}_{T,a,b}$ such that the distance between any two points $p_i$ and $p_j$ is at most $\ell$.
Also note that this expression does not depend on $K$ and is finite~\cite{bsv2021}.
This is true also in the remaining three cases $a>b$, $a=\emptyset$, and $b=\emptyset$.

Let $E(p_1,\dots,p_k|p_1,p_k)$ be the event that, for fixed $p_1$ and $p_k$, there is a $k$-hole $\{p_1,\dots,p_k\}$ in $S$ with the diameter  $|p_1p_k|$.
Then the limit $\lim_{n \to \infty}n^{k-2}\Pr[E(p_1,\dots,p_k \mid p_1,p_k)]$ equals
\begin{align*}\sum_{a \in \{2,\dots,k-1\}} \Pr[E_{\emptyset,a,\emptyset}&(p_1,\dots,p_k \mid p_1,p_k)] + \sum_{b \in \{2,\dots,k-1\}}\Pr[E_{\{2,\dots,k-1\},\emptyset,b}(p_1,\dots,p_k \mid p_1,p_k)]\\ &+\sum_{\substack{T \subseteq \{2,\dots,k-1\}\\ T,\overline{T} \neq \emptyset}}\sum_{a \in \overline{T}}\sum_{b \in T} \Pr[E_{T,a,b}(p_1,\dots,p_k \mid p_1,p_k)],
\end{align*}
which also does not depend on the convex body~$K$.

The expected number of $k$-holes in $S$ determined by $p_1$ and $p_k$ with the diameter $|p_1p_k|$ converges to $n^{k-2}\frac{\Pr[E(p_1,\dots,p_k \mid p_1,p_k)]}{n^{k-2}} = \Pr[E(p_1,\dots,p_k \mid p_1,p_k)]$.
Considering all possible pairs $\{p_1,p_k\}$, 
it follows from the dominated convergence theorem that $\lim_{n \to \infty}n^{-2}EH_{2,k}^K(n) =  \frac{1}{2}
\lim_{n \to \infty}n^{k-2}
\Pr[E(p_1,\dots,p_k \mid p_1,p_k)]$, which does not depend on~$K$, $p_1$ and~$p_k$.
This finishes the proof of Theorem~\ref{thm:independence}.

\paragraph{Acknowledgements}
We thank Sophia and Jonathan Rau for helping with the computation in the proof of Theorem~\ref{thm:4holes}.
We would also like to thank to the anonymous reviewers for several helpful comments and for raising the question about the existence of the limits $\lim_{n \to \infty}EH^K_{d,k}(n)$ for general values of $d$ and $k$ and pointing out the connections to the Hyperplane conjecture.

\bibliography{references}
\bibliographystyle{alphaabbrv-url}

\end{document}